\documentclass{article}

\usepackage[preprint]{main}

\usepackage[utf8]{inputenc} 
\usepackage[T1]{fontenc}    
\usepackage{hyperref}       
\usepackage{url}            
\usepackage{booktabs}       
\usepackage{amsfonts}       
\usepackage{nicefrac}       
\usepackage{microtype}      
\usepackage{xcolor}         
\usepackage{tikz}           
\usepackage{bbm}
\title{Stochastic Rounding Increases Small Singular Values}

\author{%
  Linkai Ma\thanks{Equal contribution.} \\
  Department of Computer Science\\
  Purdue University\\
  \texttt{ma856@purdue.edu} \\
  \And
  Tingzhou Yu \footnotemark[1] \\
  Department of Mathematics \\
  University of Alberta \\
  \texttt{tingzho1@ualberta.ca} \\
  \AND
  Petros Drineas \\
  Department of Computer Science \\
  Purdue University \\
  \texttt{pdrineas@purdue.edu} \\
}

\input{Definitions}

\begin{document}

\maketitle

\begin{abstract}
Over the past half-dozen years, stochastic rounding (SR) has regained significant attention as a quantization scheme for low-precision floating-point arithmetic, with applications spanning numerical analysis and modern machine learning systems. Recent work has shown that SR acts as an implicit regularizer by increasing the smallest singular value of extremely tall-and-thin (or, symmetrically, short-and-fat) matrices. In this work, we substantially sharpen and extend this understanding in two directions. First, we show that the regularization effect of SR is not restricted to extreme aspect ratio regimes: it persists for matrices with constant aspect ratio. Second, we demonstrate that SR does not merely regularize the smallest singular value, but instead lifts entire clusters of singular values at the tail of the spectrum. Together, these results provide a more general characterization of stochastic rounding as a spectral regularizer, revealing that its effects extend beyond extremal aspect ratios and act on a broader portion of the singular value spectrum.
\end{abstract}

\section{Introduction}

In recent years, the rapid scaling of machine learning models, such as deep neural networks and Large Language Models, has driven an increasing reliance on low-precision arithmetic to meet the computational and memory demands of modern AI workloads~\citep{micikevicius2018mixed,dettmers2022gpt3}. Quantization techniques~\citep{gholami2022survey,jin2024comprehensive,gupta2025effective}, which map high-precision values to low-bit representations, have become important for both training and inference. Among these techniques, stochastic rounding (SR) has recently re-emerged as a powerful alternative to traditional deterministic rounding schemes~\citep{gupta2015deep,xia2024stochastic,ozkara2025stochastic}, owing to its unbiasedness and favorable error propagation properties.

SR was introduced in the Numerical Analysis community in the 1950s in~\cite{forsythe1959rounding,vonneumann1947numerical} and was mostly forgotten until recently. In the early 2020s~\citep{ipsen2020probabilistic,connolly2021stochastic} highlighted its role in mitigating numerical error accumulation, followed by a growing body of work suggesting that stochastic rounding has non-trivial structural effects on the matrices it is applied to. In particular, recent theoretical results~\citep{dexter2025stochastic} demonstrate that SR can act as an implicit regularizer: when applied entrywise to a matrix, it increases the smallest singular value and promotes full-rank structure with high probability. This phenomenon has important, yet under-explored, implications for downstream tasks, including regression, optimization, and deep learning, where spectral properties of matrices play a central role in numerical stability and generalization.

However, despite this progress, our current theoretical understanding of SR-induced regularization remains incomplete in at least two key aspects. \textit{First,} existing results are largely confined to extreme aspect ratio regimes, such as very tall-and-thin matrices (or, symmetrically, very short-and-fat matrices), where randomness can accumulate across a large number of rows or columns. While these settings are analytically convenient, they do not capture the near-square matrices that frequently arise in modern machine learning pipelines, including self-attention weights, fully-connected and convolutional weight matrices, etc.
%
\textit{Second,} prior analyses focus almost exclusively on the smallest singular value, treating regularization as a one-dimensional phenomenon. In practice, however, many applications, particularly in deep learning, depend not only on the extremal singular values but also on the structure of the spectral tail, where clusters of small singular values encode fine-grained information. Indeed, recent studies highlight that the transition of this spectral tail into a heavy-tailed distribution serves as a critical signature of implicit self-regularization, often dictating the network's generalization capacity~\citep{mahoney2019traditional,martin2021implicit,martin2021predicting,he2026alphadecay}.

These limitations raise a research question: \textit{To what extent does stochastic rounding reshape the entire tail of the singular value spectrum, and does its regularization effect persist beyond extreme aspect ratio regimes?} In this work, we provide a substantially sharper and more general understanding of stochastic rounding as a spectral regularizer. We show that the regularization effect of SR is not restricted to pathological matrix geometries, but instead extends to matrices with constant aspect ratio. Moreover, we move beyond the smallest singular value and demonstrate that SR systematically regularizes clusters of singular values at the tail of the spectrum, revealing a richer and more structured form of implicit regularization than previously recognized.

\subsection{The Uniformly Dithered Quantizer and Stochastic Rounding}\label{sec:intro_quantizer}

Prior to formally stating our contributions, we define our notion of SR via the so-called Uniformly Dithered Quantizer. Indeed, SR admits several concrete instantiations. Throughout this work we focus on the \emph{uniformly dithered quantizer} as a canonical and widely used choice~\citep{saha2023matrix}. Our analysis extends to any SR scheme producing independent, unbiased, bounded entrywise noise. Fixing this instantiation lets us state quantitative bounds in terms of a single resolution parameter $\rho$.

We adopt the same definition as in Section~2.1 of~\citep{saha2023matrix} and define the quantizer for a scalar $x$. Given a dynamic range $R$ (with $|x| \le R$) and a bit-budget $B$, the quantizer uses $M = 2^B$ distinct quantization points $q_1, q_2, \ldots, q_M$, evenly spaced between $-R$ and $R$ with resolution 
\begin{equation}\label{eqn:pd1}
\rho = \frac{2R}{M-1}.   
\end{equation}
The quantization operation $Q_{\rho}(x)$ is stochastic: for a value $x \in [q_k, q_{k+1}]$, the output is determined probabilistically by
\begin{equation}\label{eqn:Q_rho}
    Q_{\rho}(x) =
    \begin{cases}
        q_{k+1} & \text{with probability } p = \frac{(x-q_k)}{\rho}, \\
        q_k     & \text{with probability } 1-p.
    \end{cases}
\end{equation}
The quantizer enjoys several properties: it is unbiased; its quantization error is bounded by the resolution $\rho$; the variance of the quantization error is bounded by $\nicefrac{\rho^2}{4}$; and the quantization error is sub-gaussian (see Appendix~\ref{app:quantizer} for details).

\subsection{Our contributions}

We summarize our two main contributions below. Let $\Ab \in \mathbb{R}^{n \times d}$, with $n > d$, be a matrix that is stochastically rounded (or quantized) by applying the Uniformly Dithered Quantizer of the previous section with resolution $\rho$ to form the matrix $\Abtil \in \mathbb{R}^{n \times d}$ as follows:
\begin{equation}
    \Abtil_{ij} = Q_{\rho}(\Ab_{ij}),\ \ \text{for all}\ i,j.
\end{equation}

Our \textit{first contribution} (Section~\ref{sec:smallsv}) is to resolve the main open problem of~\cite{dexter2025stochastic} by showing that the regularization effect of SR applies to matrices with constant aspect ratio.~\cite{dexter2025stochastic} established that SR increases the smallest singular value primarily in \textit{very tall-and-thin or short-and-fat} regimes. We prove that this phenomenon persists even when the number of rows and columns are of the same order, substantially broadening the applicability of SR-based regularization. Roughly speaking, we prove that the smallest singular value of $\Abtil$, denoted as $\sigma_{\min}(\Abtil)$, is, with high probability, at least\footnote{Up to lower order terms. Here $\rho$ is the resolution of the quantizer and $\nu$ is a parameter measuring the column-wise variance of $\Ab - \Abtil$, which effectively measures the amount of available stochasticity when $\Ab$ is rounded to form $\Abtil$. See Section~\ref{noise-model} for the exact definition of $\nu$.} 
\begin{equation*}
    \sigma_{\min}(\Abtil) \gtrsim \rho \sqrt{n \nu},
\end{equation*}
for $d \leq c_0\cdot n$, for some constant $c_0$.~\cite{dexter2025stochastic} proved that the same bound holds,\footnote{Up to lower order terms.} with comparably high probability, only when $d \ll n^{\nicefrac{1}{4}}$. Precisely, Corollary~4.2 of~\citet{dexter2025stochastic} requires $d = o\rbr{(\nicefrac{n}{\log n})^{\nicefrac{1}{4}}}$. At $n = 10^6$ this caps $d$ by $16$, enforcing an aspect ratio of at least $6\times 10^4$. Additionally, their lower order terms decay very slowly, and~\citet{dexter2025stochastic} themselves noted that $n > 10^{50}$ is required just to drop the lower order term factor below $\nicefrac{1}{2}$.

Our new bound demonstrates that stochastic rounding is a general-purpose spectral regularizer, rather than one that relies on extreme dimensional imbalances. See Theorem~\ref{thm:combined} and Corollary~\ref{cor:rectangular} for precise statements of our results and further comparisons with the bounds of~\cite{dexter2025stochastic}.

Our \textit{second contribution} (Section~\ref{sec:smallcluster}) moves beyond the focus on the smallest singular value and shows that stochastic rounding regularizes entire clusters of small singular values. Specifically, we establish that SR \textit{lifts} not just the minimum singular value, but a non-trivial portion of the spectral tail. This reveals that SR induces a spectral tail form of regularization, affecting the geometry of the low-energy subspace of the matrix rather than a single direction. This contribution provides a more accurate description of how stochastic rounding interacts with high-dimensional data and model representations. More precisely, suppose $\Ab$ has a large spectral gap at index $k$, i.e., $\sigma_k (\Ab) \gg \sigma_{k+1}(\Ab)$. We analyze the \textbf{expected excess tail energy} in the spectrum of $\Abtil$, namely
\begin{equation}\label{eqn:Tdef}
\Tcal := \mathbb{E}\left[\sum_{i = k+1}^{d} \sigma_i(\Abtil)^2 \right] - \sum_{i = k+1}^{d} \sigma_i(\Ab)^2.   
\end{equation}
Here $\sigma_i(\cdot)$ denotes the $i$-th singular value of the respective matrix. Theorem~\ref{thm:smallcluster-main} and Corollary~\ref{cor:smallcluster} show that, excluding lower order terms, the \textbf{expected excess tail energy} is lower bounded by the smallest $(d-k)$ column variances of the noise matrix $\Abtil-\Ab$. Our analysis uses resolvents and contour integration to isolate the relevant set of singular values of the matrix, as well as other tools from random matrix theory and high-dimensional probability. An alternative interpretation of our result is that SR injects noise that is diffused across directions, preventing concentration in low-dimensional subspaces and thereby lifting not just a single singular value, but an entire portion of the spectrum. 

We conclude by highlighting connections between our results and the AI/ML literature. (See also Appendix~\ref{sxn:priorwork} for a discussion of relevant prior work on SR.) Recent work has explained that the spectral structure of weight and embedding matrices plays a central role in optimization and generalization in modern machine learning systems, particularly through the emergence of heavy-tailed spectra and implicit self-regularization phenomena~\citep{mahoney2019traditional,martin2021implicit,martin2021predicting}. Our results suggest that stochastic rounding can actively reshape this structure, providing a new mechanism for influencing spectral geometry. While quantization noise is typically viewed as harmful, we show that stochastic rounding introduces structured noise that can improve spectral properties of matrices, in contrast to conventional deterministic quantization schemes~\citep{gholami2022survey,jin2024comprehensive}. Our findings move towards providing a theoretical foundation for the empirical success of stochastic rounding in low-precision training~\citep{ozkara2025stochastic,kwun2025lotion,zhaodirect}.

\section{Background}

\subsection{Notation}

Let $[k]$ denote the set of positive integers ${1, 2, \cdots, k}$. Let $\mathrm{i}$ denote the imaginary unit. Scalars are denoted by lowercase letters (e.g., $x$, $y$), vectors by bold lowercase letters (e.g., $\xb$, $\yb$), and matrices by bold uppercase letters (e.g., $\Ab$, $\Bb$). We use $\Ib_k$ to denote the $k \times k$ identity matrix (or simply $\Ib$ when the dimension is clear from context). For a matrix $\Ab \in \mathbb{R}^{n \times d}$ with $n > d$, we denote its column space (range) by $\mathcal{R}(\Ab)$. We use $\Pb_{\Ab}$ to denote the orthogonal projector onto the column space of $\Ab$. We write $\twonorm{\Ab}$ for the spectral norm and $\|\Ab\|_F$ for the Frobenius norm.
Let $\Ab \in \mathbb{R}^{n \times d}$ with $n>d$ have rank $r$. We denote its thin Singular Value Decomposition (SVD) by
$\Ab = \Ub \Sigmab \Vb^\top,$
where $\Ub \in \mathbb{R}^{n \times r}$, $\Sigmab \in \mathbb{R}^{r \times r}$, and $\Vb \in \mathbb{R}^{d \times r}$. We use $\ub_i$ (respectively, $\vb_i$) to denote the $i$-th left (respectively, right) singular vector of $\Ab$, and $\sigma_i$ to denote its $i$-th singular value. Similarly, for the perturbed matrix
$\Abtil = \Ab + \Eb$
(with $\Eb \in \mathbb{R}^{n \times d}$), we denote its thin SVD by $\Abtil = \Ubtil \Tilde{\Sigmab} \Vbtil^\top$. Again, $\ubtil_i$ (respectively, $\vbtil_i$) denotes the $i$-th left (respectively, right) singular vector of $\Abtil$, and $\tilde{\sigma}_i$ denotes its $i$-th singular value. We will focus on the asymptotic scaling of terms rather than their exact magnitudes. Therefore, we do not keep track of constants closely; absolute constants such as $C, c_1, c_2,$ etc., may be absorbed or change value from line to line without explicit mention.

\subsection{Noise Model and Perturbation Structure}\label{noise-model}
Applying $Q_{\rho}$~\eqref{eqn:Q_rho} entrywise to $\Ab$ yields the stochastically rounded matrix
\begin{equation}
    \Abtil = Q_{\rho} (\Ab) = \Ab + \Eb,
    \label{eq:SR-model}
\end{equation}
where the entries $\Eb_{ij}$ satisfy several properties: They are independent across $(i,j)$; they have zero mean; they are bounded in absolute value by $\rho$; their variance is at most $\nicefrac{\rho^2}{4}$; and they are sub-gaussian. See properties~\eqref{prop:quant-unbiased}--\eqref{prop:quant-subg} in Appendix~\ref{app:quantizer} for details.

\paragraph{Column variance matrix.} Because the entries of $\Eb$ are centered and independent, the expected Gram of $\Eb$ is diagonal:
$\Db := \Ex{\Eb^\top \Eb} = \Diag(\nu_1, \dots, \nu_d)$, where $\nu_j := \sum_{i=1}^{n} \Var(\Eb_{ij})$.
The quantity $\nu_j$ is the \emph{column variance} of column $j$. In particular, $\nu_j \le n\rho^2/4$ by~\eqref{prop:quant-var}. We denote the decreasing
rearrangement of $\nu_1,\dots,\nu_d$ as $\nu_1^\downarrow \ge \cdots \ge \nu_d^\downarrow$.

\paragraph{Normalized variance parameters.} Our analysis of $\sigma_d(\Abtil)$ will invoke the following normalized variance quantities, rescaled by the resolution $\rho$:
\begin{equation}\label{def:nu}
    \nu := \frac{1}{n\rho^2}\min_{1\le j\le d}\{\nu_j\},\quad
    \bar\nu := \frac{1}{n\rho^2}\max_{1\le j\le d}\{\nu_j\},\quad
    \mu := \frac{1}{\rho^2}\max_{1\le i\le n}\sum_{j=1}^{d} \Var(\Eb_{ij}),\quad \kappa := \max\cbr{\mu,\; n\bar\nu}.
\end{equation}

\paragraph{Gram matrices and the perturbation.} We denote the Gram matrices of $\Ab$ and $\Abtil$ by $\Gb := \Ab^\top \Ab$ and $\Gbtil := \Abtil^\top \Abtil$, and the corresponding perturbation by
\begin{equation*}
    \Deltab := \Gbtil - \Gb = \Ab^\top \Eb + \Eb^\top \Ab + \Eb^\top \Eb.
\end{equation*}
Since the entries of $\Eb$ are centered, $\Ex{\Ab^\top \Eb} = \Ex{\Eb^\top \Ab} = 0$, and hence
\begin{equation*}
    \Ex{\Deltab} = \Ex{\Eb^\top \Eb} = \Db = \Diag(\nu_1, \dots, \nu_d).
\end{equation*}
We denote the $i$-th largest eigenvalue of $\Gb$ as $\lambda_i = \sigma_i^2(\Ab)$ and of $\Gbtil$ as $\tilde{\lambda}_i = \sigma_i^2(\Abtil)$.

\subsection{A Lower Bound on $\sigma_d(\Eb)$}\label{subsec:bvh}
A key ingredient for our analysis of $\sigma_d(\Abtil)$ is a sharp\footnote{Up to logarithmic terms.} lower bound on the smallest singular value of a noise matrix with independent, centered, bounded, nonhomogeneous entries, due to~\citep[Corollary~3.18]{brailovskaya2024universality}. We restate it here, specialized to our setting.

\begin{theorem}\label{prop:BVH}
Let $\Eb$ be as in equation~\eqref{eq:SR-model} and normalized variance parameters $\nu, \mu, \kappa$ be as before. Assume $\kappa \ge \log n$. Then for every $a > 0$ there exists a constant $C_a > 0$, depending only on $a$, such that
\begin{equation*}
    \pr{\sigma_d(\Eb) \ge \rho\rbr{\sqrt{n\nu} - \sqrt{\mu} - C_a \kappa^{1/3}(\log n)^{2/3}}} \ge 1 - n^{-a}.
\end{equation*}
\end{theorem}
\begin{proof}
Set $\Yb := \Eb/\rho$. Then $\Yb \in \R^{n \times d}$ has independent centered entries with $\abr{\Yb_{ij}} \le 1$. Applying~\citep[Corollary~3.18]{brailovskaya2024universality} to $\Yb$ yields the claim. See Appendix~\ref{app:BVH_proof} for details.
\end{proof}

\subsection{Resolvent and Contour Integration}\label{subsec:resolvent}
A central tool in our analysis of the small-cluster regularization is the classical contour-integral representation of spectral projectors, which was recently utilized to derive sharp matrix perturbation bounds~\citep{tran2026davis,tran2026eigenvaluestabilitynewperturbation,tran2026newmatrixperturbationbounds,tran2024matrixperturbationboundscontour,transpectral,tran2025new}. We record the essential pieces here; the cluster-bound argument later applies them to the Gram matrices $\Gb$ and $\Gbtil$ of Section~\ref{noise-model}. For simplicity, assume that $\Ab$ and $\Abtil$ both have full column rank.

\paragraph{Resolvent.} For $z \in \mathbb{C}$ outside the spectrum of $\Gb$, the \emph{resolvent} of $\Gb$ is $\Rb_\Gb(z) := (z\Ib - \Gb)^{-1}.$
Writing the eigendecomposition of the Gram matrix as $\Gb = \Vb\Lambdab\Vb^\top$ with $\Lambdab = \Diag(\lambda_1,\dots,\lambda_d)$ and $\Vb = [\vb_1,\dots,\vb_d]$, the resolvent admits the spectral decomposition
\begin{equation} \label{eq:resolvent_decomp}
    \Rb_\Gb(z) = \sum_{i=1}^{d}\frac{\vb_i\vb_i^\top}{z-\lambda_i}, \qquad \twonorm{\Rb_\Gb(z)} = \frac{1}{\dist(z,\,\operatorname{spec}(\Gb))}.
\end{equation}

\paragraph{Neumann expansion for the perturbed resolvent.} For $\Gbtil = \Gb + \Deltab$, the factorization $z\Ib - \Gbtil = (z\Ib - \Gb)(\Ib - \Rb_\Gb(z)\Deltab)$ gives $\Rb_{\Gbtil}(z) = (\Ib - \Rb_\Gb(z)\Deltab)^{-1}\Rb_\Gb(z)$. Whenever $\twonorm{\Rb_\Gb(z)\Deltab} < 1$, this unrolls into the Neumann series
\begin{equation}\label{eq:neumann}
    \Rb_{\Gbtil}(z) = \sum_{\ell=0}^{\infty}\rbr{\Rb_\Gb(z)\Deltab}^\ell \Rb_\Gb(z).
\end{equation}

\paragraph{Spectral projectors via contour integrals.} Suppose $\Gb$ has a spectral gap at index $k$, i.e., $\lambda_k > \lambda_{k+1}$, and write $g := \lambda_k - \lambda_{k+1} > 0$. Let $\Gamma \subset \mathbb{C}$ be a positively oriented closed contour that encloses $\lambda_1,\dots,\lambda_k$ and excludes $\lambda_{k+1},\dots,\lambda_d$. For any function $f$ holomorphic inside $\Gamma$, Cauchy's integral formula applied term-wise to the spectral decomposition of $\Rb_\Gb(z)$ yields
\begin{equation}\label{eq:contour-projector}
    \frac{1}{2\pi \mathrm{i}}\oint_\Gamma f(z)\,\Rb_\Gb(z)\,dz = \sum_{\ell=1}^{k} f(\lambda_\ell)\,\vb_\ell\vb_\ell^\top.
\end{equation}
Choosing $f(z) = 1$ recovers the orthogonal projector onto the top-$k$ eigenspace of $\Gb$, while $f(z) = z$ and taking traces yields
\begin{equation} \label{eq:trace_contour}
    \sum_{i=1}^{k}\sigma_i(\Ab)^2 \;=\; \sum_{i=1}^{k}\lambda_i \;=\; \tr\rbr{\sum_{i=1}^k \lambda_i \vb_i \vb_i^\ts} \;=\; \frac{1}{2\pi \mathrm{i}}\oint_\Gamma z\,\tr\rbr{\Rb_\Gb(z)}\,dz,
\end{equation}
which expresses the top-$k$ energy as a contour integral of the resolvent. The analogous formula holds for $\Gbtil$ whenever $\Gamma$ still separates its top-$k$ eigenvalues from the rest, so combining it with~\eqref{eq:neumann} represents the energy shift on the top-$k$ block as a convergent series in $\Deltab$.
\section{Regularization of the Smallest Singular Value} \label{sec:smallsv}

In this section, we present and sketch the proof of our first main contribution: A high-probability lower bound on $\sigma_d(\Abtil)$ for stochastically rounded matrices that only requires a constant aspect ratio.

Similar to~\cite{dexter2025stochastic}, we start from the inequality 
\begin{equation}\label{eq:sigma-decomp}
    \sigma_d(\Abtil) \ge\; \sigma_d(\Eb) - \twonorm{\Pb_\Ab \Eb},
\end{equation}
where $\Pb_\Ab$ is the orthogonal projector onto $\mathcal{R}(\Ab)$. We then proceed to bound the two terms on the right separately. We significantly improve the result of \eqref{eq:dexter-bound} by sharpening the lower bound on $\sigma_d(\Eb)$ and the upper bound on $\twonorm{\Pb_\Ab \Eb}$. Specifically, the Brailovskaya--van Handel bound (Theorem~\ref{prop:BVH}) already controls $\sigma_d(\Eb)$ from below; it remains to tightly control the projection $\twonorm{\Pb_\Ab \Eb}$ from above. Throughout this section, we set $r$ to be the rank of $\Ab$ and state the projection bound in terms of $r$ rather than $d$, so that the result remains sharp when $\Ab$ is rank-deficient.

%

The following lemma gives a net-based bound on the $\twonorm{\Pb_\Ab \Eb}$. We sketch its proof below. For details, see Appendix~\ref{app:PAE}.

\begin{lemma}\label{lem:PAE}
Let $\Ab, \Eb$ be as in equation~\eqref{eq:SR-model} and let $r = \rank(\Ab)$. There exist absolute constants $C_0, c_0 > 0$ such that, for every $u \ge 0$,
\begin{equation}\label{eq:PAE-bound}
    \pr{\twonorm{\Pb_\Ab \Eb} > C_0\,\rho\rbr{\sqrt{d+r} + u}} \le 2 e^{-c_0 u^2}.
\end{equation}
\end{lemma}

\begin{proof}
Let $\Ub \in \R^{n \times r}$ have orthonormal columns spanning $\mathcal{R}(\Ab)$, so that $\Pb_\Ab = \Ub\Ub^\top$ and $\twonorm{\Pb_\Ab \Eb} = \twonorm{\Ub^\top \Eb}$. Set $\Yb := \Ub^\top \Eb \in \R^{r \times d}$. For fixed $\xb \in \mathbb{S}^{d-1}$ and $\yb \in \mathbb{S}^{r-1}$, the scalar $\yb^\top \Yb \xb = (\Ub\yb)^\top \Eb \xb$ is a weighted sum of independent centered bounded entries; Hoeffding's lemma yields the sub-gaussian moment-generating bound $\Ex{\exp(\lambda \yb^\top \Yb \xb)} \le \exp(\lambda^2 \rho^2/2)$, and the Chernoff bound gives $\pr{\abr{\yb^\top \Yb \xb} > t} \le 2 e^{-t^2/(2\rho^2)}$. Choosing $1/4$-nets of $\mathbb{S}^{d-1}$ and $\mathbb{S}^{r-1}$ of cardinality at most $9^d$ and $9^r$~\citep[Corollary~4.2.11]{Vershynin_2026}, the two-net lemma~\citep[Lemma~4.4.2]{Vershynin_2026} combined with a union bound over the at most $9^{d+r}$ pairs and the choice $t = C\rho(\sqrt{d+r}+u)$ (with $C$ large enough) yields~\eqref{eq:PAE-bound}. 
\end{proof} 

Next, we derive our main theorem, lower-bounding $\sigma_d(\Abtil)$ by combining \eqref{eq:sigma-decomp} with Theorem~\ref{prop:BVH} and Lemma~\ref{lem:PAE} via a union bound.
\begin{theorem}\label{thm:combined}
Under the noise model of Section~\ref{noise-model}, suppose $\Ab \in \R^{n \times d}$ has rank $r$ and $\kappa \ge \log n$. Then for every $a > 0$ there exist a constant $C_a > 0$ depending only on $a$, an absolute constant $C_1 > 0$, and an absolute constant $c_1 > 0$ such that, for every $u \ge 0$,
\begin{equation}\label{eq:combined-bound}
    \pr{\sigma_d(\Abtil) \ge \rho\rbr{\sqrt{n\nu} - \sqrt{\mu} - C_a\,\kappa^{1/3}(\log n)^{2/3} - C_1\rbr{\sqrt{d+r}+u}}} \ge 1 - n^{-a} - 2 e^{-c_1 u^2}.
\end{equation}
\end{theorem}
\begin{proof}
By Theorem~\ref{prop:BVH}, on an event $\Ecal_1$ with $\pr{\Ecal_1} \ge 1 - n^{-a}$,
\begin{equation*}
    \sigma_d(\Eb) \ge \rho\rbr{\sqrt{n\nu} - \sqrt{\mu} - C_a\,\kappa^{1/3}(\log n)^{2/3}}.
\end{equation*}
By Lemma~\ref{lem:PAE}, on an event $\Ecal_2$ with $\pr{\Ecal_2} \ge 1 - 2 e^{-c_1 u^2}$,
\begin{equation*}
    \twonorm{\Pb_\Ab \Eb} \le C_1\,\rho\rbr{\sqrt{d+r} + u},
\end{equation*}
where $C_1 := C_0$ and $c_1 := c_0$ are the constants from Lemma~\ref{lem:PAE}. On $\Ecal_1 \cap \Ecal_2$,~\eqref{eq:sigma-decomp} gives
\begin{equation*}
    \sigma_d(\Abtil) \ge \sigma_d(\Eb) - \twonorm{\Pb_\Ab \Eb} \ge \rho\rbr{\sqrt{n\nu} - \sqrt{\mu} - C_a \kappa^{1/3}(\log n)^{2/3} - C_1\rbr{\sqrt{d+r}+u}}.
\end{equation*}
A union bound yields $\pr{\Ecal_1 \cap \Ecal_2} \ge 1 - n^{-a} - 2 e^{-c_1 u^2}$, establishing~\eqref{eq:combined-bound}.
\end{proof}

Since $r \le d$, the bound~\eqref{eq:combined-bound} immediately implies a slightly cruder but rank-free version with $\sqrt{d+r}$ replaced by $\sqrt{2d}$. Setting $u = c_1^{-1/2}\sqrt{a\log n}$ in the resulting inequality and imposing a mild lower bound on $\nu$ together with a constant aspect-ratio constraint $d \le \eta_0 n$ yields the following headline result (see also Appendix~\ref{app:rectangular}).
\begin{corollary}\label{cor:rectangular}
Let $\Ab \in \R^{n \times d}$. Fix $a > 0$ and $\nu_0 > 0$ and let $\eta_0 = \eta_0(a,\nu_0) > 0$, $c_0 = c_0(\nu_0) > 0$, and $n_0 = n_0(a,\nu_0) \in \NN$ be constants. Suppose $\nu \ge \nu_0$, $d \le \eta_0 n$, and  $n \ge n_0$. Then
\begin{equation}\label{eq:rectangular-bound}
    \pr{\sigma_d(\Abtil) \ge c_0\,\rho\sqrt{n}} \ge 1 - 3 n^{-a}.
\end{equation}
\end{corollary}
In contrast to the $d = o\rbr{(n/\log n)^{1/4}}$ requirement of~\citet{dexter2025stochastic}, Corollary~\ref{cor:rectangular} accommodates any constant aspect ratio $d/n \le \eta_0$, thus covering slightly rectangular matrices that dominate numerical linear algebra workloads. We do emphasize that the regime where $d\approx n$ is still not covered by our bounds.

\section{Regularization of a Small Singular Value Cluster}\label{sec:smallcluster}

We now turn to our second main contribution: SR regularizes not just the smallest singular value, but a small cluster of singular values at the tail of the spectrum. Throughout this section we work in the setup of Section~\ref{noise-model} and assume that $\Ab$ has a large spectral gap at index $k$, $ \sigma_k(\Ab) \gg \sigma_{k+1}(\Ab),$ or, equivalently in terms of the Gram-matrix eigenvalues, $\lambda_k \gg \lambda_{k+1}$. We denote the spectral gap by $g := \lambda_k - \lambda_{k+1}$. Throughout this section, we assume $\Ab$ and $\Abtil$ both have full column rank. The quantity of interest is the \textbf{expected excess tail energy} $\Tcal$, as defined in~\eqref{eqn:Tdef}.

Our main results in this section are Theorem~\ref{thm:smallcluster-main} and Corollary~\ref{cor:smallcluster}. We present them at the end of the section, after outlining their proofs, as the additional notation and intuition developed in the proof sketches will help the reader better interpret and appreciate the results.

The proof of Theorem~\ref{thm:smallcluster-main} (from which Corollary~\ref{cor:smallcluster} follows) consists of four main steps.
\begin{enumerate}
    \item An expected energy identity for the entire spectrum (Theorem~\ref{thm:frobenius-energy}) reduces a lower bound on $\Tcal$ to an upper bound on the expected top-$k$ singular-value energy shift (see~\eqref{eq:tail-top-identity}).
    \item Cauchy's integral formula along a contour $\Gamma$ separating $\lambda_1,\ldots,\lambda_k$ from the rest of the spectrum, combined with the Neumann expansion of the perturbed resolvent, writes the expected top-$k$ shift as a series of contour integrals $\sum_{\ell\ge 1} I_\ell$ (Lemma~\ref{lem:topk-series}).
    \item The leading terms $I_1$ and $I_2$ are computed exactly via residue calculus and then upper bounded via concentration inequalities (Lemmas~\ref{lem:I1} and~\ref{lem:I2}); together with Step~1, $I_1$ produces the leading term $\sum_{j > k}\nu_j^\downarrow$ in the bound.
    \item The remainder term $\sum_{\ell\ge 3} I_\ell$ is bounded via triangle inequalities for contour integrals, where we use the high-probability bound on $\twonorm{\Deltab}$ from Lemma~\ref{lem:Delta-opnorm} to bound the integrand.
\end{enumerate}


As a first observation, the expected total squared singular-value mass increases after rounding. This was implicitly stated in Theorem~9 of~\citep{boutsikas2024small}.\footnote{Here we refer to the first version of this paper on arXiv.} We state it below and prove it in Appendix~\ref{app:frobenius-energy}.

\begin{theorem}\label{thm:frobenius-energy}
Let $\Ab \in \R^{n\times d}$ and $\Abtil = \Ab+\Eb$, where $\Eb$ has independent and centered entries. Then,
\begin{equation*}
    \Ex{\sum_{i=1}^{d}\sigma_i(\Abtil)^2} = \sum_{i=1}^{d}\sigma_i(\Ab)^2 + \Ex{\|\Eb\|_F^2}.
\end{equation*}
\end{theorem}


By Theorem~\ref{thm:frobenius-energy} and splitting the full singular-value energy into the top and tail parts,
\begin{equation}\label{eq:tail-top-identity}
    \Tcal = \Ex{\|\Eb\|_F^2} - \rbr{\Ex{\sum_{i=1}^{k}\sigma_i(\Abtil)^2} - \sum_{i=1}^{k}\sigma_i(\Ab)^2}.
\end{equation}
So any upper bound on the change in the top-$k$ singular-value energy immediately yields a lower bound on $\Tcal$.


To control the top-$k$ block we apply the contour-integral framework of Section~\ref{subsec:resolvent} along a contour tailored to the spectral gap at index $k$. Recall that the spectral gap is defined as $g := \lambda_k - \lambda_{k+1} > 0.$ Let $\Gamma \subset \mathbb{C}$ be the positively oriented closed contour formed by joining two half-circles of radius $g/2$ centered at $\lambda_1$ and $\lambda_k$ by two horizontal segments of length $\lambda_1 - \lambda_k$ at heights $\pm g/2$ (Figure~\ref{fig:contour}). By construction $\Gamma$ encloses exactly $\lambda_1, \ldots, \lambda_k$ and excludes $\lambda_{k+1}, \ldots, \lambda_d$.

\begin{figure}[ht]
    \centering
    \begin{tikzpicture}[scale=0.88] 
        \def\xnext{0.7}
        \def\xleft{2.2}
        \def\xright{8.2}
        \pgfmathsetmacro{\rad}{(\xleft-\xnext)/2}
    
        \draw[->] (-0.8,0) -- (10.8,0) node[right] {$\Re z$};
    
        \fill (\xnext,0) circle (2pt) node[below=6pt] {$\lambda_{k+1}$};
        \fill (\xleft,0) circle (2pt) node[below=6pt] {$\lambda_k$};
        \fill (3.6,0) circle (2pt);
        \fill (5.0,0) circle (2pt);
        \fill (6.4,0) circle (2pt);
        \fill (\xright,0) circle (2pt) node[below=6pt] {$\lambda_1$};
        \node at (5.7,-0.38) {$\cdots$};
        \fill (-0.1,0) circle (2pt);
        \node at (0.3,-0.38) {$\cdots$};
    
        \draw[<->] (\xnext,-0.95) -- (\xleft,-0.95);
        \node at (1.45,-1.28) {$g$};
    
        \draw[thick,blue]
            (\xleft,\rad) -- (\xright,\rad);
        \draw[thick,blue]
            (\xright,-\rad) -- (\xleft,-\rad);
        \draw[thick,blue]
            (\xleft,\rad) arc[start angle=90,end angle=270,radius=\rad];
        \draw[thick,blue]
            (\xright,-\rad) arc[start angle=-90,end angle=90,radius=\rad];
    
        \node[blue] at (4.3,1.45) {$\Gamma$};
    
        \draw[dashed] (\xleft,0) -- (\xleft,\rad);
        \draw[dashed] (\xright,0) -- (\xright,\rad);
        \node at (2.65,0.48) {$g/2$};
        \node at (8.85,0.48) {$g/2$};
    \end{tikzpicture}
    \caption{An illustration of the contour $\Gamma$.}
    \label{fig:contour}
\end{figure}
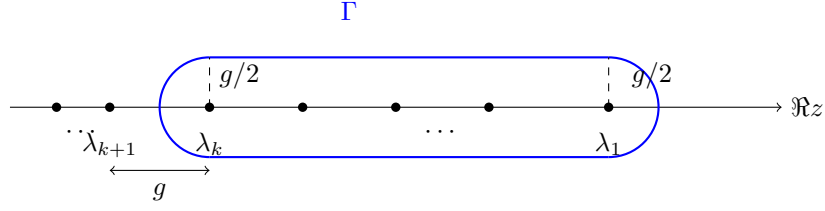
    
Every point of $\Gamma$ is at distance at least $g/2$ from $\operatorname{spec}(\Gb)$: on either half-circle $\Gamma$ has radius $g/2$, while on the horizontal segments $\abr{\Im z} = g/2$ and the eigenvalues are real. Hence by~\eqref{eq:resolvent_decomp},
\begin{equation}\label{eq:resolvent-norm-Gamma}
  \forall z \in \Gamma, \qquad  \dist(z,\,\operatorname{spec}(\Gb)) \ge \frac{g}{2}, \qquad \twonorm{\Rb_\Gb(z)} \le \frac{2}{g}.
\end{equation}

Combining the contour-integral representation~\eqref{eq:trace_contour} of the top-$k$ energy with the Neumann expansion~\eqref{eq:neumann} of the perturbed resolvent yields the following series representation of the expected top-$k$ energy shift. We state it as Lemma~\ref{lem:topk-series} and prove it in Appendix~\ref{app:topk-series}.

\begin{lemma}[Top-$k$ energy expansion]\label{lem:topk-series}
Suppose $\twonorm{\Deltab} < g/2$. Then
\begin{equation}\label{eq:topk-series}
    \Ex{\sum_{i=1}^{k}\sigma_i(\Abtil)^2} - \sum_{i=1}^{k}\sigma_i(\Ab)^2 \;=\; \sum_{\ell=1}^{\infty} I_\ell, \qquad I_\ell \;:=\; \frac{1}{2\pi \mathrm{i}}\oint_\Gamma z\,\Ex{\tr\rbr{\rbr{\Rb_\Gb(z)\Deltab}^\ell \Rb_\Gb(z)}}\,dz.
\end{equation}
\end{lemma}

The next subsection integrates and bounds the leading terms $I_1, I_2$. For the higher-order tail $\sum_{\ell\ge 3} I_\ell$, we provide an upper bound via a contour-integral argument.


We collect five lemmas: a closed-form computation of $I_1$, a residue plus concentration bound on $I_2$, a high probability upper bound of $\twonorm{\Eb}$, a high probability upper bound of $\twonorm{\Deltab}$, and a Neumann-tail bound on $\sum_{\ell\ge3} I_\ell$.

\begin{lemma}[First-order term]\label{lem:I1}
Let $\Vb_k := [\vb_1, \ldots, \vb_k] \in \R^{d \times k}$ collect the top-$k$ right singular vectors of $\Ab$. Then
\begin{equation*}
    I_1 = \tr\rbr{\Vb_k \Vb_k^\ts \Db} \le \sum_{j=1}^{k} \nu_j^\downarrow.
\end{equation*}
\end{lemma}

\begin{proof}
    Because $\Ex{\Deltab} = \Db$, residue calculus on $z/(z-\lambda_i)^2$ around $\Gamma$ collapses the contour integral to $\tr(\Vb_k \Vb_k^\ts \Db)$. The upper bound follows from Ky Fan's maximum principle applied to the diagonal matrix $\Db$. For details, see Appendix~\ref{app:I1}.
\end{proof}

The proofs of the next two lemmas are in Appendices~\ref{app:I2} and~\ref{app:E-opnorm}, respectively.

\begin{lemma}[Second-order term]\label{lem:I2}
There is an absolute constant $C > 0$ such that
\begin{equation*}
    I_2 \le \frac{C}{g}\sbr{\rho^2 \sum_{i=1}^{k}\sum_{j=k+1}^{d}(\lambda_i + \lambda_j) + \sum_{j=1}^{k}\rbr{\nu_j^\downarrow}^2 + \rho^4 n\, k(d-k)}.
\end{equation*}
\end{lemma}

\begin{proof}
    Cyclic property of trace and the spectral decomposition of $\Rb_\Gb(z)$ reduce $I_2$ to $\sum_{i=1}^k \sum_{j=k+1}^d \Ex{(\vb_i^\ts \Deltab \vb_j)^2}/(\lambda_i - \lambda_j)$ via case-by-case residue calculus. For each $(i,j)$, the linear part of $\Deltab = \Ab^\ts\Eb + \Eb^\ts\Ab + \Eb^\ts\Eb$ contributes $C\rho^2(\lambda_i + \lambda_j)$ via Hoeffding, and the quadratic part $\Eb^\ts \Eb$ contributes $(\vb_i^\ts \Db \vb_j)^2 + C\rho^4 n$ via Hanson--Wright. Summing over $(i,j)$, applying $\lambda_i - \lambda_j \ge g$, and bounding $\sum_{i \le k < j}(\vb_i^\ts \Db \vb_j)^2$ by $\sum_{j=1}^{k}(\nu_j^\downarrow)^2$ via Ky Fan's maximum principle gives the claim. 
   For details, see Appendix~\ref{app:I2}.
\end{proof}


The next lemma bounds the operator norm of the error matrix $\Eb$, which models the quantization error $\Abtil-\Ab$. See Appendix~\ref{app:E-opnorm} for its detailed proof.

\begin{lemma}[Operator norm of $\Eb$]\label{lem:E-opnorm}
Under the noise model of Section~\ref{noise-model}, there exist positive constants $C,c$ such that, for every $t\ge 0$,
\[
    \mathbb P\left\{
    \|\Eb\|_2\le C\rho(\sqrt{n+d}+t)
    \right\}
    \ge 1- 2e^{-ct^2}.
\]
Moreover, there exists a positive constant $C'$ such that
\[
    \mathbb P\left\{
    \|\Eb\|_2\le C'\rho\sqrt n
    \right\}
    \ge 1-2e^{-(n+d)}.
\]
\end{lemma}

\begin{proof}
    A standard $\varepsilon$-net argument: choose $1/4$-nets $\Ncal \subset \mathbb{S}^{n-1}$ and $\Dcal \subset \mathbb{S}^{d-1}$ with cardinalities $|\Ncal|\le 9^n$ and $|\Dcal|\le 9^d$, apply Hoeffding's inequality to the bounded sum $\xb^\ts \Eb^\ts \yb$ for each fixed pair $\xb\in \Ncal$ and $\yb\in \Dcal$, and union bound over the $9^{n+d}$ pairs. 
   For details, see Appendix~\ref{app:E-opnorm}.
\end{proof}

\begin{lemma}[Operator norm of $\Deltab$]\label{lem:Delta-opnorm}
Under the noise model of Section~\ref{noise-model}, with $n \ge d$ and $\Ab$ of full column rank, there exist positive constants $C, c_1, c_2$, such that:
\begin{equation*}
    \pr{\twonorm{\Deltab} \le C (\sigma_1(\Ab) \rho \sqrt{d} + \rho^2 n)} \ge 1 - c_1 e^{-c_2 d}.
\end{equation*}
\end{lemma}

\begin{proof}
Decompose $\twonorm{\Deltab}$ using triangle inequality. Then apply Lemma~\ref{lem:PAE} and Lemma~\ref{lem:E-opnorm}. A union bound yields the final result. For details, see Appendix~\ref{app:Delta-opnorm}.
\end{proof}

With Lemma~\ref{lem:Delta-opnorm} in hand we can define the contraction parameters that drive the Neumann tail bound below: let $C$ be the constant from Lemma~\ref{lem:Delta-opnorm} and set
\begin{equation}\label{eq:alpha-beta}
    \alpha(\rho) \;:=\; \frac{2}{g}\,C\rbr{\sigma_1(\Ab)\,\rho\sqrt{d} + \rho^2 n},
    \qquad
    \beta(\rho) \;:=\; \frac{2}{g}\rbr{2\,\sigma_1(\Ab)\,\rho\sqrt{nd} + \rho^2 nd}.
\end{equation}
On the high-probability event of Lemma~\ref{lem:Delta-opnorm}, $\twonorm{\Rb_\Gb(z)\Deltab} \le \alpha(\rho)$ uniformly in $z\in\Gamma$; on its complement, the deterministic Frobenius bound $\twonorm{\Eb}\le\|\Eb\|_F\le\rho\sqrt{nd}$ gives $\twonorm{\Rb_\Gb(z)\Deltab}\le\beta(\rho)$.

\begin{remark}
    We could alternatively use the deterministic bound $\twonorm{\Deltab} \le 2 \sigma_1(\Ab) \twonorm{\Eb} + \twonorm{\Eb}^2$ and then invoke Lemma~\ref{lem:E-opnorm}. This yields a looser bound with a higher success probability. See Appendix~\ref{sec:alter_bounds_small_cluster} for details.
\end{remark}

\begin{lemma}[Higher-order remainder]\label{lem:remainder}
Assume $\alpha(\rho) < 1$ and $\beta(\rho) < 1$ with $\alpha,\beta$ as in~\eqref{eq:alpha-beta}. Then there is an absolute constant $C > 0$ such that
\begin{equation*}
    \abr{\sum_{\ell=3}^{\infty} I_\ell} \le \frac{Cd}{g}\rbr{2(\lambda_1 - \lambda_k) + \pi g}\rbr{\lambda_1 + \frac{g}{2}}\rbr{\frac{\alpha(\rho)^3}{1 - \alpha(\rho)} + c_1 e^{-c_2 d}\,\frac{\beta(\rho)^3}{1 - \beta(\rho)}},
\end{equation*}
where $c_1, c_2$ are the constants from Lemma~\ref{lem:Delta-opnorm}.
\end{lemma}

\begin{proof}
    Let $\Gcal$ be the event of Lemma~\ref{lem:Delta-opnorm}, on which $\twonorm{\Rb_\Gb(z)\Deltab}\le \alpha(\rho)$. On $\Gcal^c$, we have $\twonorm{\Rb_\Gb(z)\Deltab}\le \beta(\rho)$. Using triangle inequality for each integral $I_\ell$ and summing the resulting geometric series completes the proof. 
   For details, see Appendix~\ref{app:remainder}.
\end{proof}


We are now ready for our main result.

\begin{theorem}\label{thm:smallcluster-main}
Assume $\alpha(\rho) < 1$ and $\beta(\rho) < 1$ with $\alpha, \beta$ as in~\eqref{eq:alpha-beta}. Then
\begin{equation*}
    \Tcal \ge \sum_{j=k+1}^{d}\nu_j^\downarrow - \frac{C}{g}\sbr{\rho^2 \sum_{i=1}^{k}\sum_{j=k+1}^{d}(\lambda_i + \lambda_j) + \sum_{j=1}^{k}\rbr{\nu_j^\downarrow}^2 + \rho^4 n\, k(d-k)} - \Rcal(\rho),
\end{equation*}
where the higher-order remainder $\Rcal(\rho)$ is the right-hand side of Lemma~\ref{lem:remainder}.
\end{theorem}

\begin{proof}
Combine the reduction~\eqref{eq:tail-top-identity} with the series~\eqref{eq:topk-series}, observing $\Ex{\|\Eb\|_F^2} = \tr(\Db) = \sum_{j=1}^{d}\nu_j$ and $\sum_{j=1}^{d}\nu_j - \sum_{j=1}^{k}\nu_j^\downarrow = \sum_{j=k+1}^{d}\nu_j^\downarrow$. Then apply Lemmas~\ref{lem:I1}, \ref{lem:I2}, and \ref{lem:remainder}. For details, see Appendix~\ref{app:smallcluster-main}.    
\end{proof}

Specializing the theorem under mild scaling assumptions yields a clean asymptotic statement. (See Appendix~\ref{app:smallcluster-cor} for details.)

\begin{corollary}\label{cor:smallcluster}
There exist absolute constants $c_1, c_2, c_3, c_g > 0$ such that, if $\rho\sqrt{nd} \le c_1\sigma_k$, $\sigma_1 \le c_2\sigma_k$, $d \ge c_3\sqrt n, g= \lambda_k - \lambda_{k+1} > c_g \lambda_k$, then there is a constant $C > 0$ such that
\begin{equation*}
    \Tcal \ge \sum_{j=k+1}^{d}\nu_j^\downarrow - C\rbr{k(d-k)\,\rho^2 + \frac{k\,n^2 \rho^4}{\sigma_k^2} + \frac{d^{5/2}\rho^3}{\sigma_k}}.
\end{equation*}
\end{corollary}


\begin{remark}\label{remark:rectangle}
When $n,d,\nicefrac{n}{d}$ are all sufficiently large, the leading term $\sum_{j=k+1}^{d}\nu_j^\downarrow$ dominates. Equivalently, the \textbf{expected excess tail energy} is captured by the smallest $(d-k)$ column variance of $\Eb$. In the special case $k = d-1$, the leading term of this bound matches that of~\citep{dexter2025stochastic}. See Appendix~\ref{app:remark_rectangle}.
\end{remark}

Specializing further to the square case $d=n$ removes one assumption and simplifies the bound. (See Appendix~\ref{app:smallcluster-square} for details.)

\begin{corollary}\label{cor:smallcluster-square}
There exist absolute constants $c_1, c_2, c_g > 0$ such that, if $d = n, \rho n \le c_1\sigma_k, \sigma_1 \le c_2\sigma_k$ and $g>c_g \lambda_k$, then there is a constant $C > 0$ such that
\begin{equation*}
    \Tcal \ge \sum_{j=k+1}^{n}\nu_j^\downarrow - C\rbr{k(n-k)\,\rho^2 + \frac{n^{5/2}\rho^3}{\sigma_k}}.
\end{equation*}
\end{corollary}

\begin{remark}\label{remark:square}
When $n,\nicefrac{n}{k}$ are all sufficiently large, the leading term $\sum_{j=k+1}^{d}\nu_j^\downarrow$ dominates, equivalently, the \textbf{expected excess tail energy} is captured by the smallest $(d-k)$ column variances of $\Eb$. See Appendix~\ref{app:remark_square} for a detailed statement and precise conditions.

\end{remark}
\section{Conclusions}\label{sec:conclusion}

In this work, we developed a sharper understanding of stochastic rounding (SR) as a spectral regularizer. 
Our results leave certain regimes uncovered. Corollary~\ref{cor:rectangular} 
requires a constant aspect ratio $d \le \eta_0 n$ and therefore does not 
cover the near-square regime $d \approx n$. Corollary~\ref{cor:smallcluster} relies on restrictive assumptions, namely a strong spectral gap, a bounded ratio $\nicefrac{\sigma_1}{\sigma_k}$, and a small perturbation level 
$\rho\sqrt{nd} \le c_1 \sigma_k$. While restrictive, these assumptions enable what is, to our knowledge, the first rigorous analysis showing that stochastic rounding regularizes not only the smallest singular value, but entire clusters at the tail of the spectrum. Understanding whether these assumptions can be relaxed, whether they are information-theoretically necessary, and how frequently they arise in practical machine learning workloads are important questions for future work. The corollary further imposes  $d \ge c_3 \sqrt{n}$, thus excluding extremely tall-and-thin matrices. However, we believe this restriction is an artifact of our analysis, since the looser bound in Corollary~\ref{cor:smallcluster_alter} dispenses with it. Finally, the nature of our work is theoretical and further empirical validations, beyond~\citep{gupta2015deep,ali2024stochastic}, of the effects of SR in downstream deep learning applications would be useful.

We conclude by noting that our analysis suggests a unifying perspective in which SR injects noise that diffuses across directions, preventing concentration in low-dimensional subspaces and strengthening weak spectral modes. 
%
%
An important direction for future theoretical work is to understand how these spectral effects demonstrate themselves when SR is used to convert matrices to low precision in optimization problems, ranging from simple least-squares regression to linear and semidefinite programming. A much more ambitious and complex goal would be to understand the impact of SR to generalization error in modern machine learning systems, particularly in low-precision training.

\newpage
\bibliographystyle{plainnat}
\bibliography{references}

\newpage
\appendix

\section{Appendix}

\subsection{Prior Work}\label{sxn:priorwork}

Stochastic rounding (SR) is a probabilistic approach to rounding that has proven effective in large-scale computations and low-precision arithmetic. Despite its illustrious beginnings in the 1950s~\citep{vonneumann1947numerical,green1950november,forsythe1959rounding}, SR has been largely overlooked by the numerical analysis community. Crucially, the random nature of SR promotes error cancellation rather than deterministic error accumulation. In many core numerical applications, ranging from floating point summation~\citep{hallman2023precision,de2025error}, inner product computation~\citep{ipsen2020probabilistic,el2023stochastic} to polynomial evaluation via Horner's algorithm~\citep{el2023stochastic}, the accumulation of rounding errors under SR can be modeled as a sum of zero-mean random variables. Consequently, these errors concentrate around zero substantially faster than the worst-case bounds characteristic of deterministic rounding. Furthermore, SR inherently prevents stagnation problems typical of traditional deterministic rounding modes~\citep{connolly2021stochastic,croci2023effects}, where numerous tiny updates to a comparatively large quantity are completely obliterated. 
%
%
For an extensive overview of the historical development, hardware implementations, and probabilistic error analysis of stochastic rounding, we refer the interested reader to the survey~\citep{croci2022stochastic}. The SIAM News article “Stochastic Rounding 2.0”\footnote{The article can also be found \href{https://www.siam.org/publications/siam-news/articles/stochastic-rounding-20-with-a-view-towards-complexity-analysis/}{here.}}~\citep{drineas2024stochastic} offers a broader view, connecting the use of low-precision arithmetic with emerging complexity analyses and proposing that quantization-aware techniques may shape the future theoretical landscape of numerical algorithms for ML/AI.

The aforementioned theoretical advantages of SR extend into the optimization of modern machine learning architectures, particularly in the large-scale training of Large Language Models (LLMs) under stringent numerical precision constraints. As the parameter counts of foundational models scale into the billions, the adoption of resource-efficient, low-precision arithmetic formats, such as BF16, FP8, and emerging 4-bit structures like MXFP4, has become critical~\citep{kalamkar2019study,micikevicius2022fp8,mishra2025recipes,tseng2025training}. In these very low-precision regimes, traditional deterministic rounding inherently introduces biased quantization noise, which distorts gradient calculations, hinders algorithmic convergence, and frequently causes training stagnation~\citep{xia2024stochastic}. Stochastic rounding, conversely, functions as an unbiased quantizer and yields accurate gradient estimates that effectively preserve critical update signals~\citep{zhaodirect}. Recent theoretical and empirical analyses demonstrate that SR fundamentally smooths the optimization landscapes of quantized objectives~\citep{kwun2025lotion} and acts as an implicit regularizer, yielding convergence bounds that seem to mitigate quantization errors~\citep{xia2025convergence,ozkara2025stochastic}.

We now discuss in more detail two prior works that are particularly relevant to this paper: A probabilistic analysis of stochastic rounding as an implicit regularizer of the smallest singular value, and a deterministic perturbation analysis for clusters of small singular values.

\paragraph{Stochastic rounding as an implicit regularizer.}
\citet{dexter2025stochastic} established the first probabilistic lower bound on $\sigma_d(\Abtil)$ after SR. Specialized to the uniformly dithered quantizer so that their entrywise bound $\mathcal{R}$ coincides with our resolution $\rho$, their main result (Theorem~4.1) states that if the entries of $\Eb$ are independent, centered, satisfy $\abr{\Eb_{ij}}\le\rho$, and have minimal normalized column variance $\nu$ as defined in Section~\ref{noise-model}, and if $n \ge 836$, then with probability at least $1 - n^{-c} - 2d^2/n^2$,
\begin{equation}\label{eq:dexter-bound}
    \sigma_d(\Abtil) \ge \rho\sqrt{n}\rbr{\sqrt{\nu} - \varepsilon_{n,d}}, \qquad \varepsilon_{n,d} := \sqrt{\tfrac{d}{n}} + 2d^2\sqrt{\tfrac{\log n}{n}} + \tfrac{C(\log n)^{2/3}}{n^{1/30}}\rbr{\tfrac{d}{n}}^{1/54}.
\end{equation}

Corollary~4.2 of \citet{dexter2025stochastic} requires $d = o\rbr{(n/\log n)^{1/4}}$. At $n = 10^6$ this caps $d$ by approximately $16$, enforcing an aspect ratio of at least $6\times 10^4$. The term $\nicefrac{(\log n)^{2/3}}{n^{1/30}}$ also decays slowly, and~\citet{dexter2025stochastic} themselves noted that $n > 10^{50}$ is required just to drop this factor below $1/2$.


\paragraph{Deterministic bounds for small singular value clusters.}
Motivated by the empirical observation that downcasting a matrix to lower arithmetic precision tends to lift its smallest singular values, \citet[Theorem~3.5]{boutsikas2024small} obtained a deterministic lower bound for a \emph{cluster} of small singular values under perturbation. Specialized to the present setting (Assumption~3.1 of~\citealp{boutsikas2024small}), let $\Ab \in \R^{n \times d}$ with $n \ge d$ have $\rank(\Ab) \ge d - r$ for some $r \ge 1$. Write the full SVD of $\Ab$ as $\Ab = \Ub\Sigmab\Vb^\ts$ with $\Ub \in \R^{n\times n}$, $\Vb \in \R^{d\times d}$ orthogonal and $\Sigmab \in \R^{n\times d}$ diagonal, and partition commensurately
\begin{equation*}
    \Sigmab = \begin{bmatrix} \Sigmab_1 & \mathbf{0} \\ \mathbf{0} & \Sigmab_2 \\ \mathbf{0} & \mathbf{0} \end{bmatrix}, \qquad
    \Eb = \Ub \begin{bmatrix} \Eb_{11} & \Eb_{12} \\ \Eb_{21} & \Eb_{22} \\ \Eb_{31} & \Eb_{32} \end{bmatrix} \Vb^\ts,
\end{equation*}
where $\Sigmab_1 \in \R^{(d-r)\times(d-r)}$ is nonsingular diagonal and $\Sigmab_2 \in \R^{r\times r}$ is diagonal. If $1/\twonorm{\Sigmab_1^{-1}} > 4\twonorm{\Eb}$ and $\twonorm{\Sigmab_2} < \twonorm{\Eb}$, then for all $1 \le j \le r$,
\begin{equation}\label{eq:boutsikas-bound}
    \sigma_{d-r+j}(\Ab + \Eb)^2 \ge \lambda_j\rbr{\Eb_{32}^\ts \Eb_{32} + (\Sigmab_2 + \Eb_{22})^\ts(\Sigmab_2 + \Eb_{22}) - \Rb_3} - r_4,
\end{equation}
where $\Rb_3$ and $r_4$ collect third-order and fourth-and-higher order terms in $\Eb$.

Motivated by this result, our second contribution examines the same setting, namely a cluster of small singular values separated from the rest by a spectral gap. However, we approach the problem from the perspective of stochastic rounding. Instead of providing a deterministic lower bound for each singular value in the small cluster, we analyze the \textbf{expected excess tail energy} (as defined in~\eqref{eqn:Tdef}) directly. 
%
%
We show that the leading term of the error corresponds to the sum of the $r$ smallest column variances of the noise matrix $\Eb$, namely\footnote{Up to lower-order terms.} 
$\sum_{j=d-r+1}^d \nu_j^\downarrow$, using the notation of Section~\ref{noise-model}. This provides a much more interpretable characterization than the bound given in \eqref{eq:boutsikas-bound}.

\subsection{The Uniformly Dithered Quantizer}\label{app:quantizer}

We adopt the same definition as in Section~2.1 of~\citep{saha2023matrix}. We first define the quantizer for a scalar $x$. Given a dynamic range $R$ (with $|x| \le R$) and a bit-budget $B$, the quantizer uses $M = 2^B$ distinct quantization points $q_1, q_2, \dots, q_M$, evenly spaced between $-R$ and $R$ with resolution $\rho = \frac{2R}{M-1}$.

The quantization operation $Q_{\rho}(x)$ is stochastic: for a value $x \in [q_k, q_{k+1}]$, the output is determined probabilistically by
\begin{equation}
    Q_{\rho}(x) =
    \begin{cases}
        q_{k+1} & \text{with probability } p, \\
        q_k     & \text{with probability } 1-p,
    \end{cases}
\end{equation}
where $p = \frac{x-q_k}{\rho}$.

Before stating the properties of the quantizer, we recall the \emph{sub-gaussian norm} of a real-valued random variable $X$~\citep[Definition~2.6.4]{Vershynin_2026}:
\begin{equation}\label{def:subn}
    \norm{X}_{\psi_2} := \inf\cbr{t > 0 : \Ex{\exp(X^2/t^2)} \le 2}.
\end{equation}

The quantizer has the following four statistical properties:
\begin{itemize}
    \item \textbf{Unbiased}: the quantizer is unbiased,
    \begin{equation}
        \Ex{Q_{\rho}(x)} = x.
        \label{prop:quant-unbiased}
    \end{equation}
    \item \textbf{Bounded Error}: the quantization error is bounded by the resolution,
    \begin{equation}
        \abr{Q_{\rho}(x) - x} \le \rho.
        \label{prop:quant-bounded}
    \end{equation}
    \item \textbf{Bounded Variance}: the variance of the quantization error is bounded,
    \begin{equation}
        \Ex{(Q_{\rho}(x) - x)^2} \le \frac{\rho^2}{4}.
        \label{prop:quant-var}
    \end{equation}
    \item \textbf{Sub-Gaussian}: the quantization error is sub-gaussian,
    \begin{equation}
        \norm{Q_{\rho}(x) - x}_{\psi_2} \le C\rho,
        \label{prop:quant-subg}
    \end{equation}
    for an absolute constant $C = 1/\sqrt{\log 2}$.
\end{itemize}

\begin{proof}
    Properties~\eqref{prop:quant-unbiased} and~\eqref{prop:quant-var} are established in Lemma~C.1 of~\citep{saha2023matrix}; property~\eqref{prop:quant-bounded} is immediate from the definition; property~\eqref{prop:quant-subg} is proven in Appendix~\ref{app:quant-subg}.
\end{proof}
\subsection{Proof of the Sub-Gaussian Property of the Quantization Error}
\label{app:quant-subg}

Here we prove property~\eqref{prop:quant-subg}: the quantization error $Q_{R,B}(x) - x$ is sub-gaussian with $\norm{Q_{R,B}(x) - x}_{\psi_2} \le C\rho$.

\begin{proof}
Recall that the sub-gaussian norm is defined in~\eqref{def:subn}.
If $\abr{X}\le\rho$, let $t=\frac{\rho}{\sqrt{\log 2}}$, we get
\begin{equation}
    \frac{X^2}{t^2}\le \log 2,
\end{equation}
and therefore
\begin{equation}
    \Ex{\exp(X^2/t^2)}\le 2.
\end{equation}
Hence
\begin{equation}
    \norm{X}_{\psi_2}\le \frac{\rho}{\sqrt{\log 2}}\le C\rho.
\end{equation}
In particular, applying this to $X = Q_{R,B}(x) - x$, which satisfies $\abr{X} \le \rho$ by~\eqref{prop:quant-bounded}, we obtain
\begin{equation}
    \norm{Q_{R,B}(x) - x}_{\psi_2}\le C\rho.
\end{equation}
Applying this to each entry of $\Eb$ gives $\norm{\Eb_{ij}}_{\psi_2} \le C\rho$ for all $i, j$.
\end{proof}

\subsection{Proof of Theorem~\ref{prop:BVH}}\label{app:BVH_proof}
\begin{proof}
Recall that $\rho$ is the quantization resolution. Set $\Yb := \Eb/\rho$. Then $\Yb \in \R^{n \times d}$ has independent centered entries with $\abr{\Yb_{ij}} \le 1$.  For each entry
$\Ab_{ij}$, let $q_{ij}$ be the lower grid point such that
$\Ab_{ij}\in [q_{ij},q_{ij}+\rho]$, and set
\[
    p_{ij}:=\frac{\Ab_{ij}-q_{ij}}{\rho}\in[0,1].
\]
Under the uniformly dithered quantizer,
\[
    \Abtil_{ij}=q_{ij}+\rho B_{ij},
    \qquad B_{ij}\sim \mathrm{Bernoulli}(p_{ij}),
\]
with the variables $B_{ij}$ independent over $(i,j)$. Hence
\[
    \Eb_{ij}=\Abtil_{ij}-\Ab_{ij}
    =\rho(B_{ij}-p_{ij}),
\]
and therefore, with $\Yb:=\Eb/\rho$,
\[
    \Yb_{ij}=B_{ij}-p_{ij}.
\]
In particular, the entries of $\Yb$ are independent, centered, with variance $p_{ij}(1-p_{ij})$.
Thus $\Yb$ is a nonhomogeneous centered Bernoulli matrix. Applying
\citep[Corollary~3.18]{brailovskaya2024universality} to $\Yb$ with variance
parameters
\[
    \min_{1\le j\le d}\sum_{i=1}^n \Var(\Yb_{ij}) = n\nu,
    \qquad
    \max_{1\le i\le n}\sum_{j=1}^d \Var(\Yb_{ij}) = \mu,
\]
and
\[
    \max\left\{
        \max_{1\le i\le n}\sum_{j=1}^d \Var(\Yb_{ij}),
        \max_{1\le j\le d}\sum_{i=1}^n \Var(\Yb_{ij})
    \right\}
    =
    \kappa
\]
yields the claim.
\end{proof}
\subsection{Proof of Lemma~\ref{lem:PAE}}\label{app:PAE}

The proof uses two standard tools from high-dimensional probability.

\begin{lemma}[Two-net lemma, {\citealp[Lemma~4.4.2]{Vershynin_2026}}]\label{lem:two-net}
Let $\Yb \in \R^{d_1 \times d_2}$ and $\varepsilon \in [0, 1/2)$. For any $\varepsilon$-nets $\mathcal{N}_1 \subset \mathbb{S}^{d_1 - 1}$ and $\mathcal{N}_2 \subset \mathbb{S}^{d_2 - 1}$,
\begin{equation*}
    \twonorm{\Yb} \le \frac{1}{1 - 2\varepsilon}\max_{\xb \in \mathcal{N}_2,\, \yb \in \mathcal{N}_1} \abr{\yb^\top \Yb \xb}.
\end{equation*}
\end{lemma}

\begin{lemma}[Hoeffding's lemma, {\citealp[Lemma~2.2]{Con_ineq_book}}]\label{lem:hoef}
Let $X$ be a real-valued random variable with $\Ex{X} = 0$ and $a \le X \le b$ almost surely. Then for all $\lambda \in \R$,
\begin{equation*}
    \Ex{\exp(\lambda X)} \le \exp\rbr{\frac{\lambda^2 (b-a)^2}{8}}.
\end{equation*}
\end{lemma}
\begin{proof}[Proof of Lemma~\ref{lem:PAE}]
Let $r := \rank(\Ab)$ and let $\Ub \in \R^{n \times r}$ have orthonormal columns spanning $\mathcal{R}(\Ab)$, so that $\Pb_\Ab = \Ub\Ub^\top$ and $\Ub^\top \Ub = \Ib_r$. Hence
\begin{equation*}
    \twonorm{\Pb_\Ab \Eb} = \twonorm{\Ub\Ub^\top \Eb} = \twonorm{\Ub^\top \Eb}.
\end{equation*}
Set
\begin{equation*}
    \Yb := \Ub^\top \Eb \in \R^{r \times d}.
\end{equation*}
We now bound $\twonorm{\Yb}$. Fix $\xb \in \mathbb{S}^{d-1}$ and $\yb \in \mathbb{S}^{r-1}$, and write $\ub := \Ub\yb \in \R^n$, so that $\twonorm{\ub} = 1$. Denote the $j$-th column of $\Eb$ by $\Eb_{\cdot j}$ and define
\begin{equation*}
    \xi_j := \ub^\top \Eb_{\cdot j} = \sum_{i=1}^{n} \ub_i \, \Eb_{ij}.
\end{equation*}
Because the entries of $\Eb$ are independent and centered, the variables $\xi_1, \dots, \xi_d$ are independent and centered.

We claim that each $\xi_j$ for $1 \le j \le d$ is sub-gaussian. Indeed, since $\ub_i \Eb_{ij}$ is centered and lies in the interval $[-\rho \abr{\ub_i}, \rho \abr{\ub_i}]$, Lemma~\ref{lem:hoef} gives, for every $\lambda \in \R$,
\begin{equation*}
    \Ex{\exp(\lambda \, \ub_i \Eb_{ij})} \le \exp\rbr{\frac{\lambda^2 \rho^2 \ub_i^2}{2}}.
\end{equation*}
By independence over $i$ and $\sum_{i=1}^{n} \ub_i^2 = 1$,
\begin{equation*}
    \Ex{\exp(\lambda \xi_j)}
    = \prod_{i=1}^{n} \Ex{\exp(\lambda \, \ub_i \Eb_{ij})}
    \le \exp\rbr{\frac{\lambda^2 \rho^2}{2} \sum_{i=1}^{n} \ub_i^2}
    = \exp\rbr{\frac{\lambda^2 \rho^2}{2}}.
\end{equation*}
Thus each $\xi_j$ is centered and sub-gaussian with proxy $\rho^2$.

Now
\begin{equation*}
    \yb^\top \Yb \xb = \yb^\top \Ub^\top \Eb \xb = \ub^\top \Eb \xb = \sum_{j=1}^{d} \xb_j \xi_j.
\end{equation*}
Applying the same moment-generating-function computation, now to the independent variables $\xi_j$ with $\sum_{j=1}^{d} \xb_j^2 = 1$, we obtain
\begin{equation*}
    \Ex{\exp(\lambda \yb^\top \Yb \xb)}
    = \prod_{j=1}^{d} \Ex{\exp(\lambda \xb_j \xi_j)}
    \le \exp\rbr{\frac{\lambda^2 \rho^2}{2} \sum_{j=1}^{d} \xb_j^2}
    = \exp\rbr{\frac{\lambda^2 \rho^2}{2}}.
\end{equation*}

By Markov's inequality, \begin{align}
    \pr{\yb^\top \Yb \xb > t} &= \pr{\exp \rbr{\lambda \yb^\top \Yb \xb} > \exp \rbr{\lambda t}} \le \frac{\Ex{\exp(\lambda \yb^\top \Yb \xb)}}{ \exp \rbr{\lambda t}}\le \exp \rbr{  \frac{\lambda^2 \rho^2}{2} - \lambda t}.
\end{align}
Optimizing in $\lambda$ yields:
\begin{equation}
    \pr{\yb^\top \Yb \xb > t} \le \exp\rbr{-\frac{t^2}{2 \rho^2}}.
\end{equation}
Bounding the other side of the tail symmetrically and taking a union bound, we get,
for every fixed $\xb \in \mathbb{S}^{d-1}$, $\yb \in \mathbb{S}^{r-1}$ and every $t \ge 0$,
\begin{equation}\label{eq:chernoff-PAE}
    \pr{\abr{\yb^\top \Yb \xb} > t} \le 2 \exp\rbr{-\frac{t^2}{2 \rho^2}}.
\end{equation}

Choose $1/4$-nets $\mathcal{N}_d \subset \mathbb{S}^{d-1}$ and $\mathcal{N}_r \subset \mathbb{S}^{r-1}$ with $\abr{\mathcal{N}_d} \le 9^d$ and $\abr{\mathcal{N}_r} \le 9^r$~\citep[Corollary~4.2.11]{Vershynin_2026}. By Lemma~\ref{lem:two-net} with $\varepsilon = 1/4$ and a union bound applied to~\eqref{eq:chernoff-PAE},
\begin{equation*}
    \pr{\twonorm{\Yb} > 2t}
    \le \pr{\max_{\xb \in \mathcal{N}_d, \, \yb \in \mathcal{N}_r} \abr{\yb^\top \Yb \xb} > t}
    \le 2 \cdot 9^{d+r} \exp\rbr{-\frac{t^2}{2 \rho^2}}.
\end{equation*}
Now choose
\begin{equation*}
    t := C \rho \rbr{\sqrt{d+r} + u}
\end{equation*}
with $C > 0$ large enough. Since $(a+b)^2 \ge a^2 + b^2$ for $a, b \ge 0$,
\begin{equation*}
    \frac{t^2}{2 \rho^2} \ge C_1 (d+r) + C_2 u^2
\end{equation*}
for some positive constants $C_1, C_2$. Choosing $C_1 > \log 9$ yields
\begin{equation*}
    2 \cdot 9^{d+r} \exp\rbr{-\frac{t^2}{2 \rho^2}} \le 2 e^{-c_0 u^2}
\end{equation*}
for some absolute $c_0 > 0$. This proves
\begin{equation*}
    \pr{\twonorm{\Pb_\Ab \Eb} = \twonorm{\Yb} > 2 C \rho \rbr{\sqrt{d+r} + u}} \le 2 e^{-c_0 u^2}.
\end{equation*}
\end{proof}

\subsection{Proof of Corollary~\ref{cor:rectangular}}\label{app:rectangular}

\begin{proof}
We first verify that the condition $\kappa\ge \log n$ in
Theorem~\ref{thm:combined}.
Since $\bar\nu\ge\nu\ge\nu_0$, for sufficiently large $n$
\[
    \kappa=\max\{\mu,n\bar\nu\}\ge n\bar\nu\ge n\nu_0\ge \log n.
\]
Since $r = \rank(\Ab) \le d$, we have $\sqrt{d+r} \le \sqrt{2d}$, and Theorem~\ref{thm:combined} yields the rank-free relaxation
\begin{equation*}
    \pr{\sigma_d(\Abtil) \ge \rho\rbr{\sqrt{n\nu} - \sqrt{\mu} - C_a\,\kappa^{1/3}(\log n)^{2/3} - C_1\sqrt{2}\rbr{\sqrt{d}+u}}} \ge 1 - n^{-a} - 2 e^{-c_1 u^2}.
\end{equation*}
Set $u := c_1^{-1/2}\sqrt{a\log n}$ in this inequality, so that $2 e^{-c_1 u^2} = 2 n^{-a}$ and the total failure probability is at most $3 n^{-a}$. Absorbing $\sqrt{2}$ into the constant, on an event of probability at least $1 - 3n^{-a}$,
\begin{equation}\label{eq:rect-step1}
    \sigma_d(\Abtil) \ge \rho\rbr{\sqrt{n\nu} - \sqrt{\mu} - C_a\,\kappa^{1/3}(\log n)^{2/3} - \widetilde{C}_1\sqrt{d} - C_1'\sqrt{\log n}},
\end{equation}
where $\widetilde{C}_1 := C_1 \sqrt{2}$ and $C_1' := C_1 \sqrt{2} c_1^{-1/2}\sqrt{a}$.

We now verify that each subtracted term on the right-hand side is asymptotically dominated by the leading $\sqrt{n\nu}$ term under the assumptions of the corollary.

\paragraph{Lower bound on the leading term.} From $\nu \ge \nu_0$, $\sqrt{n\nu} \ge \sqrt{n\nu_0}$.

\paragraph{Bound on $\sqrt{\mu}$.} Since $\Var(\Eb_{ij}) \le \rho^2/4$ by~\eqref{prop:quant-var}, the definition $\mu := \rho^{-2}\max_i \sum_j \Var(\Eb_{ij})$ gives $\mu \le d/4$. Combined with $d \le \eta_0 n$, this yields $\sqrt{\mu} \le \tfrac{1}{2}\sqrt{\eta_0 n}$.

\paragraph{Bound on the $\kappa^{1/3}(\log n)^{2/3}$ term.} By the same variance bound, $\bar\nu \le 1/4$, so $n\bar\nu \le n/4$. Similarly $\mu \le d/4 \le n/4$. Hence $\kappa := \max\cbr{\mu, n\bar\nu} \le n/4$, and $\kappa^{1/3}(\log n)^{2/3} \le (n/4)^{1/3}(\log n)^{2/3} = o(\sqrt{n})$ as $n \to \infty$.

\paragraph{Bound on $\sqrt{d}$.} Under $d \le \eta_0 n$, $\sqrt{d} \le \sqrt{\eta_0 n}$.

\paragraph{Bound on $\sqrt{\log n}$.} Trivially $\sqrt{\log n} = o(\sqrt{n})$.

Combining these bounds, the right-hand side of~\eqref{eq:rect-step1} is at least
\begin{equation*}
    \rho\rbr{\sqrt{n\nu_0} - \tfrac{1}{2}\sqrt{\eta_0 n} - \widetilde{C}_1\sqrt{\eta_0 n} - o(\sqrt{n})} \;=\; \rho\sqrt{n}\rbr{\sqrt{\nu_0} - \rbr{\tfrac{1}{2} + \widetilde{C}_1}\sqrt{\eta_0} - o(1)}.
\end{equation*}
Choose $\eta_0 = \eta_0(a, \nu_0) > 0$ small enough that $\rbr{\tfrac{1}{2} + \widetilde{C}_1}\sqrt{\eta_0} \le \tfrac{1}{2}\sqrt{\nu_0}$, and choose $n_0 = n_0(a, \nu_0)$ large enough that the $o(1)$ term is at most $\tfrac{1}{4}\sqrt{\nu_0}$ for all $n \ge n_0$. Then on the same event, $\sigma_d(\Abtil) \ge c_0\,\rho\sqrt{n}$ with $c_0 := \tfrac{1}{4}\sqrt{\nu_0}$, establishing~\eqref{eq:rectangular-bound}.
\end{proof}
\subsection{Proof of Theorem~\ref{thm:frobenius-energy}}\label{app:frobenius-energy}

\begin{proof}
For any $\Mb \in \R^{n\times d}$, $\sum_{i=1}^{d}\sigma_i(\Mb)^2 = \|\Mb\|_F^2 = \tr(\Mb^\top \Mb)$. Therefore,
\begin{align*}
    \Ex{\sum_{i=1}^{d}\sigma_i(\Abtil)^2}
    &= \Ex{\|\Abtil\|_F^2} = \Ex{\tr\rbr{(\Ab+\Eb)^\top(\Ab+\Eb)}}\\
    &= \tr(\Ab^\top \Ab) + \Ex{\tr(\Ab^\top \Eb)} + \Ex{\tr(\Eb^\top \Ab)} + \Ex{\tr(\Eb^\top \Eb)}.
\end{align*}
Since the entries of $\Eb$ are centered, $\Ex{\tr(\Ab^\top\Eb)} = \Ex{\tr(\Eb^\top\Ab)} = 0$. Hence
\begin{equation*}
    \Ex{\sum_{i=1}^{d}\sigma_i(\Abtil)^2} = \tr(\Ab^\top\Ab) + \Ex{\tr(\Eb^\top\Eb)} = \sum_{i=1}^{d}\sigma_i(\Ab)^2 + \Ex{\|\Eb\|_F^2}.
\end{equation*}
\end{proof}

\subsection{Proof of Lemma~\ref{lem:topk-series}}\label{app:topk-series}

\begin{proof}
\paragraph{Step 1: Top-$k$ energy of $\Ab$ as a contour integral.}
Applied to $\Gb$, the trace contour-integral identity~\eqref{eq:trace_contour} gives
\begin{equation}\label{eq:topk-A}
    \sum_{i=1}^{k}\sigma_i(\Ab)^2 \;=\; \frac{1}{2\pi \mathrm{i}}\oint_\Gamma z\,\tr\rbr{\Rb_\Gb(z)}\,dz.
\end{equation}

\paragraph{Step 2: $\Gamma$ also separates the top-$k$ eigenvalues of $\Gbtil$.}
By Weyl's inequality, $\abr{\lambda_i(\Gbtil) - \lambda_i(\Gb)} \le \twonorm{\Deltab} < g/2$ for every $i$. Since each point of $\Gamma$ is at distance at least $g/2$ from $\operatorname{spec}(\Gb)$, the contour $\Gamma$ continues to enclose exactly $\lambda_1(\Gbtil),\dots,\lambda_k(\Gbtil)$ and to exclude $\lambda_{k+1}(\Gbtil),\dots,\lambda_d(\Gbtil)$. Applying~\eqref{eq:trace_contour} to $\Gbtil$ along the same $\Gamma$,
\begin{equation}\label{eq:topk-Atil}
    \sum_{i=1}^{k}\sigma_i(\Abtil)^2 \;=\; \frac{1}{2\pi \mathrm{i}}\oint_\Gamma z\,\tr\rbr{\Rb_{\Gbtil}(z)}\,dz.
\end{equation}

\paragraph{Step 3: Subtraction.}
Subtracting~\eqref{eq:topk-A} from~\eqref{eq:topk-Atil} and using linearity of the trace,
\begin{equation}\label{eq:topk-shift}
    \sum_{i=1}^{k}\sigma_i(\Abtil)^2 - \sum_{i=1}^{k}\sigma_i(\Ab)^2 \;=\; \frac{1}{2\pi \mathrm{i}}\oint_\Gamma z\,\tr\rbr{\Rb_{\Gbtil}(z) - \Rb_\Gb(z)}\,dz.
\end{equation}

\paragraph{Step 4: Neumann expansion.}
Combining the hypothesis $\twonorm{\Deltab}<g/2$ with the resolvent-norm bound~\eqref{eq:resolvent-norm-Gamma},
\begin{equation}\label{eq:RGD-bound}
    \twonorm{\Rb_\Gb(z)\Deltab} \;\le\; \twonorm{\Rb_\Gb(z)}\,\twonorm{\Deltab} \;\le\; \tfrac{2}{g}\,\twonorm{\Deltab} \;<\; 1, \qquad z\in\Gamma.
\end{equation}
Hence the Neumann series~\eqref{eq:neumann} converges and
\begin{equation}\label{eq:resolvent-diff}
    \Rb_{\Gbtil}(z) - \Rb_\Gb(z) \;=\; \sum_{\ell=1}^{\infty}\rbr{\Rb_\Gb(z)\Deltab}^\ell \Rb_\Gb(z).
\end{equation}

\paragraph{Step 5: Interchange integral and series}

Substituting the Neumann series~\eqref{eq:neumann} into~\eqref{eq:topk-shift} yields:
\begin{align}
     \sum_{i=1}^{k}\sigma_i(\Abtil)^2 - \sum_{i=1}^{k}\sigma_i(\Ab)^2 \;&=\; \frac{1}{2\pi \mathrm{i}}\oint_\Gamma z\,\tr\rbr{\sum_{\ell=1}^{\infty}\rbr{\Rb_\Gb(z)\Deltab}^\ell \Rb_\Gb(z)}\,dz \\
     &= \; \frac{1}{2\pi \mathrm{i}}\oint_\Gamma \sum_{\ell=1}^{\infty} \,z\,\tr\rbr{\rbr{\Rb_\Gb(z)\Deltab}^\ell \Rb_\Gb(z)}\,dz.
\end{align}

Denote $f_m(z) := \sum_{\ell = 1}^m z\,\tr\rbr{\rbr{\Rb_\Gb(z)\Deltab}^\ell \Rb_\Gb(z)}.$ Using $\abr{\tr(\Mb)}\le d\twonorm{\Mb}$, we get:
\begin{align}
    \abr{f_m(z)} \le \abr{z}  d  \sum_{\ell = 1}^m \twonorm{\rbr{\Rb_\Gb(z)\Deltab}^\ell \Rb_\Gb(z)}.
\end{align}

By~\eqref{eq:resolvent-norm-Gamma} and triangle inequality,
\begin{equation}
    \abr{f_m(z)} \le \abr{z}\, d \, \frac{2}{g} \sum_{\ell = 1}^m \twonorm{\Rb_\Gb(z)\Deltab} ^\ell,
\end{equation}

Since $\twonorm{\Rb_\Gb(z)\Deltab} < 1$, the geometric series is finite:
\begin{equation}
    \sum_{\ell = 1}^m \twonorm{\Rb_\Gb(z)\Deltab} ^\ell \leq \sum_{\ell = 1}^\infty \twonorm{\Rb_\Gb(z)\Deltab} ^\ell = \frac{\twonorm{\Rb_\Gb(z)\Deltab}}{1-\twonorm{\Rb_\Gb(z)\Deltab}}.
\end{equation}

Hence, 
\begin{equation}
    \abr{f(z)} \le h(z) := \abr{z}\, d \, \frac{2}{g} \frac{\twonorm{\Rb_\Gb(z)\Deltab}}{1-\twonorm{\Rb_\Gb(z)\Deltab}}.
\end{equation}

By the boundedness of $\Gamma$, $h(z)$ is clearly integrable. The Lebesgue's dominated convergence theorem then gives:
\begin{equation}
     \sum_{i=1}^{k}\sigma_i(\Abtil)^2 - \sum_{i=1}^{k}\sigma_i(\Ab)^2 \;= \; \frac{1}{2\pi \mathrm{i}}\sum_{\ell=1}^{\infty}\oint_\Gamma  \,z\,\tr\rbr{\rbr{\Rb_\Gb(z)\Deltab}^\ell \Rb_\Gb(z)}\,dz.
\end{equation}

Finally, taking expectation yields:
\begin{equation*}
    \Ex{\sum_{i=1}^{k}\sigma_i(\Abtil)^2} - \sum_{i=1}^{k}\sigma_i(\Ab)^2 \;=\; \sum_{\ell=1}^{\infty} I_\ell, \qquad I_\ell = \frac{1}{2\pi \mathrm{i}}\oint_\Gamma z\,\Ex{\tr\rbr{\rbr{\Rb_\Gb(z)\Deltab}^\ell \Rb_\Gb(z)}}\,dz.
\end{equation*}
\end{proof}

\subsection{Proof of Lemma~\ref{lem:I1}}\label{app:I1}

\begin{proof}
By linearity and $\Ex{\Deltab}=\Db$ (Section~\ref{noise-model}),
\begin{equation*}
    I_1 = \frac{1}{2\pi \mathrm{i}}\oint_\Gamma z\,\tr\rbr{\Rb_\Gb(z)\Db\Rb_\Gb(z)}\,dz = \tr\rbr{\frac{1}{2\pi \mathrm{i}}\oint_\Gamma z\,\Rb_\Gb(z)^2\,dz \;\Db}.
\end{equation*}
Writing $\Vb = [\vb_1,\dots,\vb_d]$,
\begin{equation*}
    \Rb_\Gb(z)^2 = \Vb\diag\rbr{(z-\lambda_1)^{-2},\dots,(z-\lambda_d)^{-2}}\Vb^\ts.
\end{equation*}
Since
\begin{equation}
    \operatorname*{Res}_{z=\lambda_i}\frac{z}{(z-\lambda_i)^2} = \left.\tfrac{d}{dz}z\right|_{z=\lambda_i}=1,
\end{equation}
Cauchy's integral formula gives
\begin{equation}
    \frac{1}{2\pi \mathrm{i}}\oint_\Gamma \frac{z}{(z-\lambda_i)^2}dz = \mathbbm{1}_{i\le k}.
\end{equation}
Therefore 
\begin{equation}
    \frac{1}{2\pi \mathrm{i}}\oint_\Gamma z\,\Rb_\Gb(z)^2 dz = \Vb_k\Vb_k^\ts,
\end{equation}
with 
\begin{equation}
    \Vb_k=[\vb_1,\dots,\vb_k],
\end{equation}
and
\begin{equation*}
    I_1 = \tr\rbr{\Vb_k\Vb_k^\ts \Db} = \tr\rbr{\Vb_k^\ts \Db \Vb_k}.
\end{equation*}
Ky Fan's maximum principle~\citep[Theorem 1]{fan1949theorem} gives
\begin{equation*}
    \tr\rbr{\Vb_k^\ts \Db \Vb_k} \le \max_{\Qb\in\R^{d\times k}\,:\,\Qb^\ts \Qb=\Ib_k}\tr\rbr{\Qb^\ts \Db \Qb} = \sum_{j=1}^{k}\nu_j^\downarrow,
\end{equation*}
where $\nu_1^\downarrow \ge \cdots \ge \nu_d^\downarrow$ is the decreasing rearrangement of $\nu_1,\dots,\nu_d$.
\end{proof}

\subsection{Proof of Lemma~\ref{lem:I2}}\label{app:I2}

\begin{proof}
\paragraph{Step 1: Residue calculus.}
Using cyclicity and linearity of trace plus the spectral decomposition~\eqref{eq:resolvent_decomp} $\Rb_\Gb(z) = \sum_{i=1}^d\frac{\vb_i\vb_i^\ts}{z-\lambda_i}$,
\begin{align*}
    &\tr\rbr{\Rb_\Gb(z) \Deltab \Rb_\Gb(z)\Deltab \Rb_\Gb(z)}\\
    &= \tr\rbr{\Rb_\Gb(z)^2 \Deltab \Rb_\Gb(z)\Deltab}\\
    &= \tr \rbr{\sum_{i,j=1}^{d}  \frac{\vb_i \vb_i^\ts \Deltab  \vb_j \vb_j^\ts \Deltab}{(z-\lambda_i)^2(z-\lambda_j)}}\\
    &=\sum_{i,j=1}^{d} \tr \rbr{  \frac{\vb_i \vb_i^\ts \Deltab  \vb_j \vb_j^\ts \Deltab}{(z-\lambda_i)^2(z-\lambda_j)}}\\
    &= \sum_{i,j=1}^{d}\frac{(\vb_i^\ts\Deltab\vb_j)(\vb_j^\ts\Deltab\vb_i)}{(z-\lambda_i)^2(z-\lambda_j)}.
\end{align*}
Since $\Deltab$ is symmetric, $(\vb_i^\ts\Deltab\vb_j)(\vb_j^\ts\Deltab\vb_i) = (\vb_i^\ts\Deltab\vb_j)^2$, so
\begin{equation}\label{eq:I2-master}
    I_2 = \sum_{i,j=1}^{d}\Ex{(\vb_i^\ts\Deltab\vb_j)^2}\frac{1}{2\pi \mathrm{i}}\oint_\Gamma\frac{z}{(z-\lambda_i)^2(z-\lambda_j)}\,dz.
\end{equation}

We now compute the contour integral case by case.

\paragraph{Case 1: $i,j\le k.$}
Clearly both poles lie inside $\Gamma$.

If $i\neq j$, the residues are
\begin{equation}
    \operatorname*{Res}_{z=\lambda_i}\frac{z}{(z-\lambda_i)^2(z-\lambda_j)}
    =
    \left.\frac{d}{dz}\frac{z}{z-\lambda_j}\right|_{z=\lambda_i}
    =
    -\frac{\lambda_j}{(\lambda_i-\lambda_j)^2},
\end{equation}
and
\begin{equation}
    \operatorname*{Res}_{z=\lambda_j}\frac{z}{(z-\lambda_i)^2(z-\lambda_j)}
    =
    \left.\frac{z}{(z-\lambda_i)^2}\right|_{z=\lambda_j}
    =
    \frac{\lambda_j}{(\lambda_i-\lambda_j)^2},
\end{equation}
so the integral is zero. 

If $i=j\le k$, then there is a triple pole at
$z=\lambda_i$, and
\begin{equation}
    \operatorname*{Res}_{z=\lambda_i}\frac{z}{(z-\lambda_i)^3}
    =
    \frac{1}{2!}\left.\frac{d^2}{dz^2} z\right|_{z=\lambda_i}
    =
    0,
\end{equation}
so the integral is again zero.

\paragraph{Case 2: $i,j>k.$}

Both poles lie outside $\Gamma$, so the integral is zero.

\paragraph{Case 3: $i\le k<j.$}

Only the double pole at $z=\lambda_i$ lies inside
$\Gamma$, and
\begin{equation}
    \frac{1}{2\pi \mathrm{i}}\oint_\Gamma \frac{z}{(z-\lambda_i)^2(z-\lambda_j)}\,dz
    =
    \operatorname*{Res}_{z=\lambda_i}\frac{z}{(z-\lambda_i)^2(z-\lambda_j)}
    =
    \left.\frac{d}{dz}\frac{z}{z-\lambda_j}\right|_{z=\lambda_i}
    =
    -\frac{\lambda_j}{(\lambda_i-\lambda_j)^2}.
\end{equation}

\paragraph{Case 4: $j\le k<i.$}
Only the simple pole at $z=\lambda_j$ lies inside
$\Gamma$, and
\begin{equation}
    \frac{1}{2\pi \mathrm{i}}\oint_\Gamma \frac{z}{(z-\lambda_i)^2(z-\lambda_j)}\,dz
    =
    \operatorname*{Res}_{z=\lambda_j}\frac{z}{(z-\lambda_i)^2(z-\lambda_j)}
    =
    \left.\frac{z}{(z-\lambda_i)^2}\right|_{z=\lambda_j}
    =
    \frac{\lambda_j}{(\lambda_i-\lambda_j)^2}.
\end{equation}

Therefore
\begin{align}
    I_2 &=
    \sum_{i=1}^k \sum_{j=k+1}^d
    \Ex{(\vb_i^\ts \Deltab \vb_j)^2}
    \left(-\frac{\lambda_j}{(\lambda_i-\lambda_j)^2}\right) \\
    &\quad+
    \sum_{j=1}^k \sum_{i=k+1}^d
    \Ex{(\vb_i^\ts \Deltab \vb_j)^2}
    \frac{\lambda_j}{(\lambda_i-\lambda_j)^2}.
\end{align}
Relabeling indices in the second sum and using
$(\vb_i^\ts \Deltab \vb_j)^2=(\vb_j^\ts \Deltab \vb_i)^2$, we obtain
\begin{equation}\label{eq:I2-clean}
    I_2 = \sum_{i=1}^k \sum_{j=k+1}^d
    \frac{\Ex{(\vb_i^\ts \Deltab \vb_j)^2}}{\lambda_i-\lambda_j}.
\end{equation}

\paragraph{Step 2: Bounding $\Ex{(\vb_i^\ts\Deltab\vb_j)^2}$.}
For $i\ne j$, $\Deltab = \Ab^\ts \Eb + \Eb^\ts \Ab + \Eb^\ts \Eb$ and $\Ab\vb_\ell = \sigma_\ell\ub_\ell$ give
\begin{equation*}
    \vb_i^\ts\Deltab\vb_j = \sigma_i\,\ub_i^\ts \Eb \vb_j + \sigma_j\,\ub_j^\ts \Eb \vb_i + \vb_i^\ts \Eb^\ts \Eb \vb_j.
\end{equation*}
Applying the scalar inequality $(a+b+c)^2 \le 3(a^2+b^2+c^2)$ yields:\footnote{Indeed, by Cauchy--Schwarz,
$(a+b+c)^2=\bigl((a,b,c)\cdot(1,1,1)\bigr)^2\le (a^2+b^2+c^2)(1^2+1^2+1^2)=3(a^2+b^2+c^2).$}
\begin{equation}\label{eq:vidv-split}
    \Ex{(\vb_i^\ts\Deltab\vb_j)^2} \le 3\sigma_i^2\,\Ex{(\ub_i^\ts \Eb \vb_j)^2} + 3\sigma_j^2\,\Ex{(\ub_j^\ts \Eb \vb_i)^2} + 3\,\Ex{(\vb_i^\ts \Eb^\ts \Eb \vb_j)^2}.
\end{equation}

\paragraph{Linear terms.}
For unit vectors $\ub\in\R^n$, $\vb\in\R^d$, notice that:
\begin{equation}
    \ub^\ts \Eb \vb = \sum_{a=1}^n \sum_{b=1}^d \ub_a\vb_b\Eb_{ab},
\end{equation}
which is a sum of independent centered sub-gaussian random variables.

By~\eqref{prop:quant-subg}, 
\begin{equation}
    \norm{\ub_a\vb_b\Eb_{ab}}_{\psi_2}\le C\rho\abr{\ub_a\vb_b}.
\end{equation}
The general Hoeffding inequality~\citep[Theorem~2.7.3]{Vershynin_2026} gives, for some absolute constant $c$,
\begin{equation*}
    \pr{\abr{\ub^\ts \Eb \vb}\ge t} \le 2\exp\rbr{-\frac{ct^2}{\rho^2}},
\end{equation*}
since $\sum_{a,b}\ub_a^2\vb_b^2 = 1$. Therefore, for some absolute constant $C,$
\begin{equation*}
    \Ex{(\ub^\ts \Eb \vb)^2} = \int_0^\infty 2t\,\pr{\abr{\ub^\ts \Eb \vb}\ge t}\,dt \le 4\int_0^\infty t e^{-ct^2/\rho^2}\,dt = \frac{2\rho^2}{c} \le C\rho^2.
\end{equation*}

\paragraph{Quadratic term.}
Write $\xb=\vect(\Eb)\in\R^{nd}$. Then
\begin{equation*}
    \vb_i^\ts \Eb^\ts \Eb \vb_j = \xb^\ts \Bb^{ij}\xb, \qquad \Bb^{ij} := \tfrac{1}{2}\bigl((\vb_i\vb_j^\ts)\otimes \Ib_n + (\vb_j\vb_i^\ts)\otimes \Ib_n\bigr).
\end{equation*}
Indeed, by the cyclic property of trace:

\begin{equation}
    \vb_i^\ts \Eb^\ts \Eb \vb_j = \tr(\Eb^\ts \Eb \vb_j\vb_i^\ts).
\end{equation}

Recall ~\citep[Eqs.~(520)--(521)]{petersen2008matrix},
\begin{align}
    \vect(\Ab\Xb\Bb) &= (\Bb^\ts\otimes \Ab)\,\vect(\Xb), \label{eq:vec-axb}\\
    \tr(\Ab^\ts \Bb) &= \vect(\Ab)^\ts \vect(\Bb), \label{eq:trace-vec}
\end{align}

First, use~\eqref{eq:trace-vec} with $\Ab = \Eb, \Bb = \Eb \vb_j \vb_i^\ts$:
\begin{equation}
    \tr(\Eb^\ts \Eb \vb_j\vb_i^\ts) = \vect(\Eb)^\ts \vect( \Eb \vb_j \vb_i^\ts).
\end{equation}

Second, use~\eqref{eq:vec-axb} with $\Ab = \Ib_n, \Xb = \Eb, \Bb = \vb_j \vb_i^\ts$:
\begin{equation}
    \vect( \Eb \vb_j \vb_i^\ts) = \bigl((\vb_i\vb_j^\ts)\otimes\Ib_n\bigr)\vect(\Eb).
\end{equation}

Substituting the above two results, we get:
\begin{equation}
     \vb_i^\ts \Eb^\ts \Eb \vb_j =\vect(\Eb)^\ts \bigl((\vb_i\vb_j^\ts)\otimes\Ib_n\bigr)\vect(\Eb).
\end{equation}

Finally, since a quadratic form depends only on the symmetric part of its matrix:
\begin{align}
    \vb_i^\ts \Eb^\ts \Eb \vb_j &= \frac{1}{2}\vect(\Eb)^\ts \rbr{(\vb_i\vb_j^\ts)\otimes\Ib_n + \rbr{(\vb_i\vb_j^\ts)\otimes\Ib_n}^\ts}\vect(\Eb)\\
    &= \frac{1}{2}\vect(\Eb)^\ts \rbr{(\vb_i\vb_j^\ts)\otimes\Ib_n + (\vb_j\vb_i^\ts)\otimes\Ib_n}\vect(\Eb)\\
    &= \vect(\Eb)^\ts \Bb^{ij} \vect(\Eb).
\end{align}

Next, notice that since $\vb_i$ and $\vb_j$ are unit vectors, both $(\vb_i\vb_j^\ts)\otimes\Ib_n$ and $(\vb_j\vb_i^\ts)\otimes\Ib_n$ have $n$ $1$-valued singular values and $n(d-1)$ $0$-valued singular values. Therefore, 
\begin{align}
    &\twonorm{(\vb_i\vb_j^\ts)\otimes\Ib_n} = \twonorm{(\vb_j\vb_i^\ts)\otimes\Ib_n} =1 \\
    &\norm{(\vb_i\vb_j^\ts)\otimes\Ib_n}_F = \norm{(\vb_j\vb_i^\ts)\otimes\Ib_n}_F = \sqrt{n}.
\end{align}

Triangle inequality yields:
\begin{equation}
    \twonorm{\Bb^{ij}} \le 1, \quad\norm{\Bb^{ij}}_F \le \sqrt{n}.
\end{equation}

The Hanson--Wright inequality~\citep[Theorem~6.2.2]{Vershynin_2026} therefore gives, for any $t\ge 0$, there exists a positive constant $c$ such that:
\begin{equation*}
Y := \vb_i^\ts \Eb^\ts \Eb \vb_j - \Ex{\vb_i^\ts \Eb^\ts \Eb \vb_j}, \quad    \pr{\abr{Y}\ge t} \le 2\exp\sbr{-c\min\rbr{\frac{t^2}{\rho^4 n},\,\frac{t}{\rho^2}}}.
\end{equation*}

We can then calculate:
\begin{equation}
    \mathbb{E}[Y^2] = \int_0^\infty 2t\,\mathbb{P}(|Y|\ge t)\,dt.
\end{equation}

The two arguments of the minimum are equal at $t=\rho^2 n$, so
\begin{equation*}
    \Ex{Y^2} \le 4\int_0^{\rho^2 n} t e^{-ct^2/(\rho^4 n)}\,dt + 4\int_{\rho^2 n}^{\infty} t e^{-ct/\rho^2}\,dt.
\end{equation*}

For the first term, with the substitution
\begin{equation}
    u=\frac{ct^2}{\rho^4 n},
    \qquad
    du=\frac{2ct}{\rho^4 n}\,dt,
\end{equation}
we get
\begin{equation}
    \int_0^{\rho^2 n} t e^{-ct^2/(\rho^4 n)}dt
    =
    \frac{\rho^4 n}{2c}\int_0^{cn} e^{-u}\,du
    \le
    \frac{\rho^4 n}{2c}.
\end{equation}
For the second term, with
\begin{equation}
    s=\frac{ct}{\rho^2},
    \qquad
    ds = \frac{c}{\rho^2}\,dt,
\end{equation}
we obtain
\begin{equation}
    \int_{\rho^2 n}^{\infty} t e^{-ct/\rho^2}dt
    =
    \frac{\rho^4}{c^2}\int_{cn}^{\infty} s e^{-s}\,ds
    =
    \frac{\rho^4}{c^2}(cn+1)e^{-cn}.
\end{equation}
Since $cn\ge 0$ and $x+1\le e^x$ for all $x\ge 0$, we have
\begin{equation}
    (cn+1)e^{-cn}\le 1,
\end{equation}
and therefore
\begin{equation}
    \frac{\rho^4}{c^2}(cn+1)e^{-cn}
    \le
    C\rho^4.
\end{equation}

Adding the two integrals, we get:
\begin{equation}
   \Var[\vb_i^\ts \Eb^\ts \Eb \vb_j] = \Ex{\rbr{\vb_i^\ts \Eb^\ts \Eb \vb_j - \Ex{\vb_i^\ts \Eb^\ts \Eb \vb_j}}^2} \le C \rho^4 n.
\end{equation}

Since $\Ex{\Eb^\ts \Eb} = \Db$, $\Ex{\vb_i^\ts \Eb^\ts \Eb \vb_j} = \vb_i^\ts \Db \vb_j$.

Finally, the decomposition of the variance gives:

\begin{align}
    \Ex{(\vb_i^\ts \Eb^\ts \Eb \vb_j)^2} &= \Ex{\vb_i^\ts \Eb^\ts \Eb \vb_j}^2 + \Var [\vb_i^\ts \Eb^\ts \Eb \vb_j]\\
    &\le (\vb_i^\ts \Db \vb_j)^2 + C\rho^4 n.
\end{align}

Combining the linear and quadratic parts in~\eqref{eq:vidv-split},
\begin{equation}\label{eq:vidv-final}
    \Ex{(\vb_i^\ts\Deltab\vb_j)^2} \le C\sbr{\rho^2(\sigma_i^2+\sigma_j^2) + (\vb_i^\ts \Db \vb_j)^2 + \rho^4 n}.
\end{equation}

\paragraph{Step 3: Plugging into $I_2$.}
Substituting~\eqref{eq:vidv-final} into~\eqref{eq:I2-clean} and using $\lambda_i=\sigma_i^2$,
\begin{equation*}
    I_2 \le C\sum_{i=1}^{k}\sum_{j=k+1}^{d}\frac{\rho^2(\lambda_i+\lambda_j) + (\vb_i^\ts \Db \vb_j)^2 + \rho^4 n}{\lambda_i-\lambda_j}.
\end{equation*}
For the middle term, with $\Pb := \Vb_k\Vb_k^\ts$ and $\Vb_{k,\perp}=[\vb_{k+1},\dots,\vb_d]$, $\Vb_{k,\perp}\Vb_{k,\perp}^\ts = \Ib - \Pb$, we can express the double sum in matrix form:

\begin{align}
    &\sum_{i=1}^{k}\sum_{j=k+1}^{d}(\vb_i^\ts \Db \vb_j)^2 =  \|\Vb_k^\ts \Db \Vb_{k,\perp}\|_F^2 \\
    &= \tr \rbr{ \Vb_k^\ts \Db \Vb_{k,\perp} \Vb_{k,\perp}^\ts \Db \Vb_k } = \tr \rbr{\Vb_k^\ts \Db (\Ib -\Pb) \Db \Vb_k}\\
    &= \tr \rbr{\Pb \Db (\Ib -\Pb) \Db }\\
    &= \tr(\Pb\Db^2) - \tr(\Pb\Db\Pb\Db) \le \tr(\Pb\Db^2),
\end{align}

since $\tr(\Pb\Db\Pb\Db) = \tr(\Db^{1/2} \Pb \Db^{1/2}  \Db^{1/2} \Pb  \Db^{1/2}) = \|\Db^{1/2}\Pb\Db^{1/2}\|_F^2 \ge 0$. 

Since $\Db^2 = \diag(\nu_1^2,\dots,\nu_d^2)$, Ky Fan's maximum principle yields $\tr(\Pb\Db^2) \le \sum_{j=1}^{k}(\nu_j^\downarrow)^2$. 

Combining with the spectral-gap bound $\lambda_i-\lambda_j \ge g$ for $i\le k<j$,
\begin{equation*}
    I_2 \le \frac{C}{g}\sbr{\rho^2\sum_{i=1}^{k}\sum_{j=k+1}^{d}(\lambda_i+\lambda_j) + \sum_{j=1}^{k}(\nu_j^\downarrow)^2 + \rho^4 n\,k(d-k)}.
\end{equation*}
\end{proof}

\subsection{Proof of Lemma~\ref{lem:E-opnorm}}\label{app:E-opnorm}

We will use Hoeffding's inequality for sums of bounded independent random variables.

\begin{theorem}[Hoeffding's inequality, {\citealp[Theorem~2.2.6]{Vershynin_2026}}]\label{thm:hoeffding}
Let $X_1,\dots,X_N$ be independent random variables with $m_i \le X_i \le M_i$ for each $i$. Then for every $t > 0$,
\begin{equation*}
    \pr{\abr{\sum_{i=1}^{N}\rbr{X_i - \Ex{X_i}}} \ge t} \le 2\exp\rbr{-\frac{2 t^2}{\sum_{i=1}^{N}(M_i - m_i)^2}}.
\end{equation*}
\end{theorem}

\begin{proof}[Proof of Lemma~\ref{lem:E-opnorm}]
The proof is based on an $\varepsilon$-net argument; see Section~4.2 of~\citet{Vershynin_2026} for a review.

\paragraph{Step 1: Approximation.}
Fix $\varepsilon = 1/4$. By \citep[Corollary~4.2.11]{Vershynin_2026}, there exist $\varepsilon$-nets $\Ncal \subset \mathbb{S}^{n-1}$ and $\Dcal \subset \mathbb{S}^{d-1}$ with bounded cardinality:
\begin{equation*}
    \abr{\Ncal} \le 9^n, \qquad \abr{\Dcal} \le 9^d.
\end{equation*}
By Lemma~\ref{lem:two-net}, $\twonorm{\Eb}$ can be bounded using the nets as
\begin{equation}\label{eq:E-net-relaxation}
    \twonorm{\Eb} \le 2 \max_{\xb\in\mathcal N,\ \yb\in\mathcal D}
    |\xb^\top \Eb \yb|.
\end{equation}

\paragraph{Step 2: Concentration.}
Fix $\xb \in \Ncal$ and $\yb \in \Dcal$. Then
\[
    \xb^\top \Eb \yb
    =
    \sum_{i=1}^{n}\sum_{j=1}^{d}
    \xb_i\yb_j\Eb_{ij},
\]
which is a sum of independent bounded random variables. Using the uniformly dithered quantizer of Section~\ref{app:quantizer}, assuming $\Ab_{ij}\in [q_{ij}, q_{ij}+\rho]$,
\begin{equation*}
    \Eb_{ij} = \begin{cases}
        \rho(1 - p_{ij}) & \text{with probability } p_{ij},\\
        -\rho\,p_{ij} & \text{with probability } 1 - p_{ij},
    \end{cases}
\end{equation*}
where $p_{ij}:=(\Ab_{ij}-q_{ij})/\rho$. Therefore $\Eb_{ij}$ lies in the interval $[-\rho p_{ij},\,\rho(1-p_{ij})]$ of width $\rho$. Applying Theorem~\ref{thm:hoeffding},
\begin{align*}
    \pr{\abr{\xb^\top \Eb \yb} \ge t}
    &\le 2\exp\rbr{-\frac{2 t^2}{\sum_{i,j}\rho^2 \xb_i^2 \yb_j^2}}\\
    &= 2\exp\rbr{-\frac{2 t^2}{\rho^2 \rbr{\sum_{i=1}^{n}\xb_i^2}\rbr{\sum_{j=1}^{d}\yb_j^2}}}
    = 2\exp\rbr{-\frac{2 t^2}{\rho^2}}.
\end{align*}

\paragraph{Step 3: Union bound.}
By a union bound on the event $\cbr{\max_{\xb\in\Ncal,\,\yb\in\Dcal}\abr{\xb^\top \Eb \yb} \ge t} = \bigcup_{\xb\in\Ncal,\,\yb\in\Dcal}\cbr{\abr{\xb^\top \Eb \yb} \ge t}$,
\begin{align*}
    \pr{\max_{\xb\in\Ncal,\,\yb\in\Dcal}\abr{\xb^\top \Eb \yb} \ge t}
    &\le \sum_{\xb\in\Ncal,\,\yb\in\Dcal}\pr{\abr{\xb^\top \Eb \yb} \ge t}\\
    &\le 9^{n+d} \cdot 2\exp\rbr{-\frac{2 t^2}{\rho^2}}.
\end{align*}
Now choose
\[
    t=C_0\rho(\sqrt{n+d}+s),
\]
where $C_0>0$ is a sufficiently large constant. Since
\[
    (\sqrt{n+d}+s)^2\ge n+d+s^2,
\]
we obtain
\[
    2\cdot 9^{n+d}
    \exp\left(-\frac{2t^2}{\rho^2}\right)
    \le
    2e^{-cs^2}.
\]
for a positive constant $c$. Hence
\[
    \mathbb P\left\{
    \|\Eb\|_2>C_0\rho(\sqrt{n+d}+s)
    \right\}
    \le
    2e^{-cs^2}.
\]
Finally, taking $s=C'\sqrt{n+d}$ with $C'$ sufficiently large gives
\[
    \mathbb P\left\{
    \|\Eb\|_2\le C'\rho\sqrt{n+d}
    \right\}
    \ge
    1-2e^{-(n+d)}.
\]
As $n\ge d$, we have for some positive constant $C$
\[
    \mathbb P\left\{
    \|\Eb\|_2\le C\rho\sqrt n
    \right\}
    \ge
    1-2e^{-(n+d)}.
\]
\end{proof}

\subsection{Proof of Lemma~\ref{lem:Delta-opnorm}}\label{app:Delta-opnorm}

\begin{proof}
\paragraph{Step 1: Triangle inequality.}
Recall $\Deltab = \Ab^\ts \Eb + \Eb^\ts \Ab + \Eb^\ts \Eb$. Since $\twonorm{\Mb^\ts} = \twonorm{\Mb}$ for any $\Mb$, the triangle inequality gives
\begin{equation}\label{eq:Delta-triangle}
    \twonorm{\Deltab} \le \twonorm{\Ab^\ts \Eb} + \twonorm{\Eb^\ts \Ab} + \twonorm{\Eb^\ts \Eb} = 2\,\twonorm{\Ab^\ts \Eb} + \twonorm{\Eb}^2.
\end{equation}

\paragraph{Step 2: Bound on $\twonorm{\Ab^\ts \Eb}$ via Lemma~\ref{lem:PAE}.}
Using the thin SVD $\Ab = \Ub \Sigmab \Vb^\ts,$
\begin{equation*}
    \twonorm{\Ab^\ts \Eb} = \twonorm{\Vb\Sigmab\Ub^\ts \Eb} \le \twonorm{\Vb}\,\twonorm{\Sigmab}\,\twonorm{\Ub^\ts \Eb} = \sigma_1(\Ab)\,\twonorm{\Ub^\ts \Eb},
\end{equation*}
Since $\Pb_\Ab = \Ub\Ub^\ts$,
\begin{equation*}
    \twonorm{\Ub^\ts \Eb} = \twonorm{\Ub\Ub^\ts \Eb} = \twonorm{\Pb_\Ab \Eb}.
\end{equation*}
Since $\Ab$ has full column rank, $r := \rank(\Ab) = d$. Applying Lemma~\ref{lem:PAE} with $u = \sqrt{d}$ and $r = d$ gives
\begin{equation*}
    \pr{\twonorm{\Pb_\Ab \Eb} > C_0\,\rho\rbr{\sqrt{2d} + \sqrt{d}}} \le 2 e^{-c_0 d},
\end{equation*}
so on an event $\Ecal_1$ with $\pr{\Ecal_1} \ge 1 - 2 e^{-c_0 d}$,
\begin{equation}\label{eq:AE-bound}
    \twonorm{\Ab^\ts \Eb} \le C_1\,\sigma_1(\Ab)\,\rho\sqrt{d}, \qquad C_1 := C_0(1 + \sqrt{2}).
\end{equation}

\paragraph{Step 3: Bound on $\twonorm{\Eb}^2$ via Lemma~\ref{lem:E-opnorm}.}
By Lemma~\ref{lem:E-opnorm}, on an event $\Ecal_2$ with $\pr{\Ecal_2} \ge 1 - 2 e^{-(n+d)}$,
\begin{equation}\label{eq:E-square-bound}
    \twonorm{\Eb}^2 \le C_2\,\rho^2 n,
\end{equation}
where $C_2$ is the square of the constant from Lemma~\ref{lem:E-opnorm}.

\paragraph{Step 4: Union bound.}
On $\Ecal_1 \cap \Ecal_2$, combining~\eqref{eq:Delta-triangle},~\eqref{eq:AE-bound}, and~\eqref{eq:E-square-bound},
\begin{equation*}
    \twonorm{\Deltab} \le 2 C_1\,\sigma_1(\Ab)\,\rho\sqrt{d} + C_2\,\rho^2 n \le C\rbr{\sigma_1(\Ab)\,\rho\sqrt{d} + \rho^2 n},
\end{equation*}
with $C := \max\{2 C_1, C_2\}$. By a union bound,
\begin{equation*}
    \pr{\Ecal_1 \cap \Ecal_2} \ge 1 - 2 e^{-c_0 d} - 2 e^{-(n+d)} \ge 1 - 4 e^{-c_0 d},
\end{equation*}
where the last step uses $n+d \ge d$, so $e^{-(n+d)} \le e^{-d}$, and absorbs constants. Setting $c_1 := 4$ and $c_2 := \min\{c_0, 1\}$ gives
\begin{equation*}
    \pr{\twonorm{\Deltab} \le C\rbr{\sigma_1(\Ab)\,\rho\sqrt{d} + \rho^2 n}} \ge 1 - c_1 e^{-c_2 d}.
\end{equation*}
\end{proof}

\subsection{Proof of Lemma~\ref{lem:remainder}}\label{app:remainder}

\begin{proof}
For each $z\in\Gamma$, set
\begin{equation*}
    X(z) := \twonorm{\Rb_\Gb(z)\Deltab},
\end{equation*}
and let $\Gcal$ denote the high-probability event of Lemma~\ref{lem:Delta-opnorm}, on which $\twonorm{\Deltab}\le C\rbr{\sigma_1\rho\sqrt d+\rho^2 n}$, with $\pr{\Gcal^c}\le c_1 e^{-c_2 d}$.

\paragraph{Step 1: Pointwise contour bound on $\abr{I_\ell}$.}
Fix $\ell\ge 3$. Recall from~\eqref{eq:topk-series} that
\begin{equation*}
    I_\ell = \frac{1}{2\pi \mathrm{i}}\oint_\Gamma z\,\Ex{\tr\rbr{(\Rb_\Gb(z)\Deltab)^\ell\Rb_\Gb(z)}}\,dz,
\end{equation*}
so by the triangle inequality on the contour integral and Jensen's inequality on the expectation,
\begin{equation}\label{eq:Iell-step1}
    \abr{I_\ell} \le \frac{1}{2\pi}\oint_\Gamma \abr{z}\,\Ex{\abr{\tr\rbr{(\Rb_\Gb(z)\Deltab)^\ell\Rb_\Gb(z)}}}\,\abr{dz}.
\end{equation}
For any $d\times d$ matrix $\Mb$ one has $\abr{\tr(\Mb)}\le d\twonorm{\Mb}$, and submultiplicativity of the operator norm gives
\begin{equation*}
    \twonorm{(\Rb_\Gb(z)\Deltab)^\ell\Rb_\Gb(z)} \le \twonorm{\Rb_\Gb(z)\Deltab}^\ell\,\twonorm{\Rb_\Gb(z)} = X(z)^\ell\,\twonorm{\Rb_\Gb(z)}.
\end{equation*}
Substituting into~\eqref{eq:Iell-step1},
\begin{equation}\label{eq:Iell-step1-final}
    \abr{I_\ell} \le \frac{d}{2\pi}\oint_\Gamma \abr{z}\,\Ex{X(z)^\ell\,\twonorm{\Rb_\Gb(z)}}\,\abr{dz} \le \frac{d}{\pi g}\oint_\Gamma \abr{z}\,\Ex{X(z)^\ell}\,\abr{dz},
\end{equation}
where for the last inequality, we used the resolvent bound~\eqref{eq:resolvent-norm-Gamma} $\twonorm{\Rb_\Gb(z) }\le \frac{2}{g}$.

\paragraph{Step 2: Bounding $X(z)$ on $\Gcal$ and on $\Gcal^c$.}
For every $z\in\Gamma$, submultiplicativity and~\eqref{eq:resolvent-norm-Gamma} gives $X(z) \le \frac{2}{g}\twonorm{\Deltab}$.

On $\Gcal$, the bound from Lemma~\ref{lem:Delta-opnorm} therefore yields
\begin{equation*}
    X(z)\,\mathbf{1}_{\Gcal} \le \frac{2}{g}\,C\rbr{\sigma_1\rho\sqrt d+\rho^2 n} = \alpha(\rho),
\end{equation*}
with $\alpha(\rho)$ as in~\eqref{eq:alpha-beta}. 

On $\Gcal^c$ the lemma no longer applies, but the entries of $\Eb$ are bounded by $\rho$, so the deterministic Frobenius bound
\begin{equation*}
    \twonorm{\Eb}\le \|\Eb\|_F \le \rho\sqrt{nd}
\end{equation*}
still holds, and the triangle inequality on $\Deltab=\Ab^\ts\Eb+\Eb^\ts\Ab+\Eb^\ts\Eb$ gives
\begin{equation*}
    \twonorm{\Deltab} \le 2\twonorm{\Ab}\twonorm{\Eb}+\twonorm{\Eb}^2 \le 2\sigma_1\rho\sqrt{nd}+\rho^2 nd.
\end{equation*}
Multiplying by $2/g$,
\begin{equation*}
    X(z)\,\mathbf{1}_{\Gcal^c} \le \frac{2}{g}\rbr{2\sigma_1\rho\sqrt{nd}+\rho^2 nd} = \beta(\rho).
\end{equation*}

\paragraph{Step 3: Taking expectation and splitting on $\Gcal$.}
Splitting the indicator $1=\mathbf{1}_{\Gcal}+\mathbf{1}_{\Gcal^c}$ inside the expectation and using the bounds of Step~2,
\begin{align*}
    \Ex{X(z)^\ell} 
    &= \Ex{X(z)^\ell\,\mathbf{1}_{\Gcal}} + \Ex{X(z)^\ell\,\mathbf{1}_{\Gcal^c}} \\
    &\le \alpha(\rho)^\ell\,\pr{\Gcal} + \beta(\rho)^\ell\,\pr{\Gcal^c} \\
    &\le \alpha(\rho)^\ell + c_1 e^{-c_2 d}\,\beta(\rho)^\ell,
\end{align*}
where the last step uses $\pr{\Gcal}\le 1$ and $\pr{\Gcal^c}\le c_1 e^{-c_2 d}$ from Lemma~\ref{lem:Delta-opnorm}. Substituting into~\eqref{eq:Iell-step1-final},
\begin{equation*}
    \abr{I_\ell} \le \frac{d}{\pi g}\rbr{\alpha(\rho)^\ell+c_1 e^{-c_2 d}\,\beta(\rho)^\ell}\oint_\Gamma\abr{z}\,\abr{dz}.
\end{equation*}

\paragraph{Step 4: Geometric parameters of the contour $\Gamma$.}
The contour $\Gamma$ consists of two horizontal segments of length $\lambda_1-\lambda_k$ joined by two half-circles of radius $g/2$ (Figure~\ref{fig:contour}), so
\begin{equation*}
    \operatorname{len}(\Gamma) = 2(\lambda_1-\lambda_k)+\pi g.
\end{equation*}
The point of $\Gamma$ farthest from the origin is the rightmost point of the right half-circle, namely $\lambda_1+g/2$, so
\begin{equation*}
    \sup_{z\in\Gamma}\abr{z} = \lambda_1+\frac{g}{2}.
\end{equation*}
Bounding $\oint_\Gamma\abr{z}\abr{dz}\le\operatorname{len}(\Gamma)\cdot\sup_{z\in\Gamma}\abr{z}$ and substituting,
\begin{equation}\label{eq:Iell-step4}
    \abr{I_\ell} \le \frac{d}{\pi g}\rbr{2(\lambda_1-\lambda_k)+\pi g}\rbr{\lambda_1+\frac{g}{2}}\rbr{\alpha(\rho)^\ell+c_1 e^{-c_2 d}\,\beta(\rho)^\ell}.
\end{equation}

\paragraph{Step 5: Geometric summation in $\ell$.}
Under the hypothesis $\alpha(\rho),\beta(\rho)<1$, both geometric series converge and
\begin{equation*}
    \sum_{\ell=3}^{\infty}\alpha(\rho)^\ell = \frac{\alpha(\rho)^3}{1-\alpha(\rho)}, \qquad \sum_{\ell=3}^{\infty}\beta(\rho)^\ell = \frac{\beta(\rho)^3}{1-\beta(\rho)}.
\end{equation*}
Summing~\eqref{eq:Iell-step4} over $\ell\ge 3$, using $\abr{\sum_\ell I_\ell}\le\sum_\ell\abr{I_\ell}$, and absorbing $1/\pi$ into an absolute constant $C$,
\begin{equation*}
    \abr{\sum_{\ell=3}^{\infty}I_\ell} \le \frac{Cd}{g}\rbr{2(\lambda_1-\lambda_k)+\pi g}\rbr{\lambda_1+\frac{g}{2}}\rbr{\frac{\alpha(\rho)^3}{1-\alpha(\rho)}+c_1 e^{-c_2 d}\,\frac{\beta(\rho)^3}{1-\beta(\rho)}}.
\end{equation*}
\end{proof}

\subsection{Proof of Theorem~\ref{thm:smallcluster-main}}\label{app:smallcluster-main}

\begin{proof}
Let $\Gcal$ be the event of Lemma~\ref{lem:Delta-opnorm}. Under the hypotheses $\alpha(\rho)<1$ and $\beta(\rho)<1$ with $\alpha,\beta$ as in~\eqref{eq:alpha-beta}, on $\Gcal$
\begin{equation*}
    \twonorm{\Deltab}\le C\rbr{\sigma_1\rho\sqrt d+\rho^2 n} = \tfrac{g}{2}\,\alpha(\rho) < \tfrac{g}{2},
\end{equation*}
while on $\Gcal^c$ the deterministic Frobenius bound $\twonorm{\Eb}\le\|\Eb\|_F\le\rho\sqrt{nd}$ together with $\twonorm{\Deltab}\le 2\twonorm{\Ab}\twonorm{\Eb}+\twonorm{\Eb}^2$ gives
\begin{equation*}
    \twonorm{\Deltab}\le 2\sigma_1\rho\sqrt{nd}+\rho^2 nd = \tfrac{g}{2}\,\beta(\rho) < \tfrac{g}{2}.
\end{equation*}
Hence $\twonorm{\Deltab}<g/2$ almost surely, so Lemma~\ref{lem:topk-series} applies and yields the convergent series~\eqref{eq:topk-series}. Combining~\eqref{eq:topk-series} with~\eqref{eq:tail-top-identity},
\begin{equation*}
    \Tcal = \Ex{\|\Eb\|_F^2} - \sum_{\ell=1}^{\infty} I_\ell = \sum_{j=1}^{d}\nu_j - I_1 - I_2 - \sum_{\ell\ge 3} I_\ell,
\end{equation*}
where $\Ex{\|\Eb\|_F^2} = \tr(\Db) = \sum_{j=1}^{d}\nu_j$. Bounding $I_1$ from above by Lemma~\ref{lem:I1}, $I_2$ from above by Lemma~\ref{lem:I2}, and $\sum_{\ell\ge 3} I_\ell$ in absolute value by Lemma~\ref{lem:remainder}, and using $\sum_{j=1}^{d}\nu_j - \sum_{j=1}^{k}\nu_j^\downarrow = \sum_{j=k+1}^{d}\nu_j^\downarrow$,
\begin{equation*}
    \Tcal \ge \sum_{j=k+1}^{d}\nu_j^\downarrow - \frac{C}{g}\sbr{\rho^2\sum_{i=1}^{k}\sum_{j=k+1}^{d}(\lambda_i+\lambda_j) + \sum_{j=1}^{k}(\nu_j^\downarrow)^2 + \rho^4 n\,k(d-k)} - \Rcal(\rho).
\end{equation*}
\end{proof}

\subsection{Proof of Corollary~\ref{cor:smallcluster}}\label{app:smallcluster-cor}

\begin{proof}
Let $\bar c_1, \bar c_2 > 0$ denote the constants from Lemma~\ref{lem:Delta-opnorm}.

\emph{Verifying the hypotheses of Theorem~\ref{thm:smallcluster-main}.} Using $\sigma_1 \le c_2\sigma_k$, $\rho\sqrt{nd}\le c_1\sigma_k$, and $n \ge d \ge 1$,
\begin{equation*}
    \sigma_1\rho\sqrt d \le \frac{c_2\sigma_k\cdot c_1\sigma_k}{\sqrt n} \le c_1 c_2\,\sigma_k^2,\qquad
    \rho^2 n = \frac{(\rho\sqrt{nd})^2}{d} \le c_1^2\,\sigma_k^2,
\end{equation*}
\begin{equation*}
    \sigma_1\rho\sqrt{nd} \le c_1 c_2\,\sigma_k^2,\qquad
    \rho^2 nd \le c_1^2\,\sigma_k^2.
\end{equation*}
Substituting into~\eqref{eq:alpha-beta} and using $g \ge c_g\sigma_k^2$,
\begin{equation*}
    \alpha(\rho) \le \frac{2C(c_1 c_2 + c_1^2)}{c_g},\qquad
    \beta(\rho) \le \frac{2(2 c_1 c_2 + c_1^2)}{c_g}.
\end{equation*}
Choosing $c_1$ sufficiently small, we may ensure $\alpha(\rho), \beta(\rho) \le 1/2$ so that $\tfrac{1}{1-\alpha(\rho)},\ \tfrac{1}{1-\beta(\rho)} \le 2$. In particular $\alpha(\rho), \beta(\rho) < 1$, so the hypotheses of Theorem~\ref{thm:smallcluster-main} are satisfied.

\emph{Bounding the leading terms.} From $\sigma_1 \le c_2\sigma_k$, $\lambda_1 = \sigma_1^2 = O(\sigma_k^2)$ and hence $\lambda_i+\lambda_j = O(\sigma_k^2)$ for all $i\le k<j$. Therefore
\begin{equation*}
    \frac{\rho^2}{g}\sum_{i=1}^{k}\sum_{j=k+1}^{d}(\lambda_i+\lambda_j) \le C\,k(d-k)\,\rho^2.
\end{equation*}
Since each $\nu_j \le n\rho^2/4$,
\begin{equation*}
    \frac{1}{g}\sum_{j=1}^{k}(\nu_j^\downarrow)^2 \le C\frac{kn^2\rho^4}{\sigma_k^2}, \qquad \frac{\rho^4 n\,k(d-k)}{g} \le C\frac{kdn\,\rho^4}{\sigma_k^2}.
\end{equation*}

\emph{Bounding the remainder.} Plugging the definitions~\eqref{eq:alpha-beta} into Lemma~\ref{lem:remainder} together with $g=\Theta(\sigma_k^2)$ and $\lambda_1=O(\sigma_k^2)$ gives
\begin{equation}\label{eq:R-cube-form}
    \Rcal(\rho) \le \frac{Cd}{\sigma_k^4}\sbr{\rbr{\sigma_1\rho\sqrt d + \rho^2 n}^3 + \bar c_1 e^{-\bar c_2 d}\rbr{2\sigma_1\rho\sqrt{nd}+\rho^2 nd}^3}.
\end{equation}
We claim the linear-in-$\rho$ part dominates the quadratic-in-$\rho$ part inside each cube. For the second cube, using $\sigma_1\ge\sigma_k$ and $\rho\sqrt{nd}\le c_1\sigma_k$,
\begin{equation*}
    \frac{2\sigma_1\rho\sqrt{nd}}{\rho^2 nd} \;=\; \frac{2\sigma_1}{\rho\sqrt{nd}} \;\ge\; \frac{2\sigma_k}{c_1\sigma_k} \;=\; \frac{2}{c_1}.
\end{equation*}
For the first cube, write $\rho n = \rho\sqrt{nd}\cdot\sqrt{n/d}\le c_1\sigma_k\sqrt{n/d}$. Using $\sigma_1\ge\sigma_k$ and the assumption $d\ge c_3\sqrt n$,
\begin{equation*}
    \frac{\sigma_1\rho\sqrt d}{\rho^2 n} \;=\; \frac{\sigma_1\sqrt d}{\rho n} \;\ge\; \frac{\sigma_k\sqrt d}{c_1\sigma_k\sqrt{n/d}} \;=\; \frac{1}{c_1}\cdot\frac{d}{\sqrt n} \;\ge\; \frac{c_3}{c_1}.
\end{equation*}
Choosing $c_1\le 1$ and $c_3\ge 2c_1$, both ratios are $\ge 2$, so $(a+b)^3\le 4 a^3$ for $a$ the linear part and $b$ the quadratic part. Thus $(\sigma_1\rho\sqrt d+\rho^2 n)^3\le 4(\sigma_1\rho\sqrt d)^3$ and $(2\sigma_1\rho\sqrt{nd}+\rho^2 nd)^3\le 4(2\sigma_1\rho\sqrt{nd})^3$. Substituting into~\eqref{eq:R-cube-form} and using $\sigma_1\le c_2\sigma_k$,
\begin{equation*}
    \Rcal(\rho) \le C\sbr{\frac{d^{5/2}\rho^3}{\sigma_k} + \bar c_1 e^{-\bar c_2 d}\,\frac{d^{5/2} n^{3/2}\rho^3}{\sigma_k}} \le C\,\frac{d^{5/2}\rho^3}{\sigma_k},
\end{equation*}
where in the last step we absorb the exponentially small bad-event contribution. More precisely, by the assumption $d\ge c_3 \sqrt{n}, \bar c_1 e^{-\bar c_2 d}\,n^{3/2}$ vanishes as $n\rightarrow \infty$.

Substituting these bounds into Theorem~\ref{thm:smallcluster-main} proves the claim.
\end{proof}

\subsection{Proof of Remark~\ref{remark:rectangle}}\label{app:remark_rectangle}
\begin{proof}
    Indeed, assume $\sum_{j=k+1}^{d}\nu_j^\downarrow \ge (d-k) n \rho^2 \nu_0$ for some constant $\nu_0$. Applying the assumption $\rho \sqrt{n d}\le c_1 \sigma_k$ yields:
\begin{align*}
    \frac{k(d-k)\,\rho^2}{\sum_{j=k+1}^{d}\nu_j^\downarrow}\;&\le\; \frac{1}{\nu_0}\cdot\frac{k}{n}, \\
    \frac{kn^2\rho^4/\sigma_k^2}{\sum_{j=k+1}^{d}\nu_j^\downarrow}\;&\le\; \frac{c_1^2}{\nu_0}\cdot\frac{k}{d\,(d-k)}, \\
    \frac{d^{5/2}\rho^3/\sigma_k}{\sum_{j=k+1}^{d}\nu_j^\downarrow}\;&\le\; \frac{c_1}{\nu_0}\cdot\frac{d^{2}}{(d-k)\,n^{3/2}}.
\end{align*}
Note that for the 3rd ratio to vanish, we need $d = o (n^{3/4})$.

Recall the definition of $\nu$ in~\eqref{def:nu}. For the special case $k = d-1$, the leading term of our bound becomes:
\begin{align*}
    \sum_{j=k+1}^{d}\nu_j^\downarrow = \nu_d^\downarrow = \min_{1\le j\le d}\{\nu_j\} = n \rho^2 \nu = \rbr{ \rho \sqrt{n \nu}}^2.
\end{align*}
\end{proof}

\subsection{Proof of Corollary~\ref{cor:smallcluster-square}}\label{app:smallcluster-square}

\begin{proof}
We specialize Corollary~\ref{cor:smallcluster} to $d=n$. The first assumption $\rho\sqrt{nd}\le c_1\sigma_k$ reduces to $\rho n\le c_1\sigma_k$, while the second is unchanged. Substituting $d=n$ into the conclusion of Corollary~\ref{cor:smallcluster},
\begin{equation*}
    \Tcal\ge \sum_{j=k+1}^{n}\nu_j^\downarrow - C\rbr{k(n-k)\,\rho^2 + \frac{kn^2\rho^4}{\sigma_k^2}+\frac{n^{5/2}\rho^3}{\sigma_k}}.
\end{equation*}
The middle term is absorbed into the last via $\rho n\le c_1\sigma_k$:
\begin{equation*}
    \frac{k n^2 \rho^4}{\sigma_k^2} = \frac{k n \rho^3}{\sigma_k^2} \cdot n \rho \le \frac{k n \rho^3}{\sigma_k^2} \cdot c_1 \sigma_k < c_1 \frac{n^2 \rho^3}{\sigma_k} < c_1 \frac{n^{5/2}\rho^3}{\sigma_k}.
\end{equation*}
which gives the claimed bound.
\end{proof}

\subsection{Proof of Remark~\ref{remark:square}}\label{app:remark_square}

\begin{proof}
    Indeed, assume $\sum_{j=k+1}^{d}\nu_j^\downarrow \ge (n-k) n \rho^2 \nu_0$ for some constant $\nu_0$. Applying the assumption $\rho n\le c_1 \sigma_k$ yields:
\begin{equation*}
    \frac{k (n-k)\rho^2}{\sum_{j=k+1}^{n}\nu_j^\downarrow}\;\le\; \frac{1}{\nu_0}\cdot\frac{k}{n}, \qquad
    \frac{n^{5/2}\rho^3/\sigma_k}{\sum_{j=k+1}^{n}\nu_j^\downarrow}\;\le\; \frac{1}{\nu_0}\cdot\frac{n^{3/2}\rho}{(n-k)\,\sigma_k}\;\le\; \frac{c_1}{\nu_0}\cdot\frac{\sqrt n}{n-k}.
\end{equation*}
\end{proof}

\subsection{Alternative Bounds for Small Clusters} \label{sec:alter_bounds_small_cluster}
In this section we provide alternative bounds for the \textbf{expected excess tail energy} after rounding,
\begin{equation*}
    \Tcal := \Ex{\sum_{i=k+1}^{d} \sigma_i(\Abtil)^2} - \sum_{i=k+1}^{d}\sigma_i(\Ab)^2.
\end{equation*}

We start with an alternative upper bound for $\twonorm{\Deltab}$, and propagate the changes to Lemma~\ref{lem:remainder}, Theorem~\ref{thm:smallcluster-main} and Corollary~\ref{cor:smallcluster}.

\begin{lemma}[Alternative bound of $\twonorm{\Deltab}$]\label{lem:Delta-opnorm_alter}
Under the noise model of Section~\ref{noise-model}, with $n \ge d$ and $\Ab$ of full column rank, there exists a positive constant $C$, such that:
\begin{equation*}
    \pr{\twonorm{\Deltab} \le C (\sigma_1(\Ab) \rho \sqrt{n} + \rho^2 n)} \ge 1 - 2 e^{-\rbr{n+d}}.
\end{equation*}
\end{lemma}
\begin{proof}
    With triangle inequality and submultiplicity, we can bound $\twonorm{\Deltab}$ by $2 \sigma_1(\Ab) \twonorm{\Eb}+ \twonorm{\Eb}^2$. Invoking Lemma~\ref{lem:E-opnorm} completes the proof.
\end{proof}

Define another contraction parameter:
\begin{equation}\label{eq:alpha_tilde}
    \tilde{\alpha}(\rho):= \; \frac{2}{g}\,C\rbr{\sigma_1(\Ab)\,\rho\sqrt{n} + \rho^2 n}.
\end{equation}
On the high-probability event of Lemma~\ref{lem:Delta-opnorm_alter}, $\twonorm{\Rb_\Gb(z)\Deltab} \le \tilde{\alpha}(\rho)$ uniformly in $z\in\Gamma$; on its complement, the deterministic Frobenius bound $\twonorm{\Eb}\le\|\Eb\|_F\le\rho\sqrt{nd}$ gives $\twonorm{\Rb_\Gb(z)\Deltab}\le\beta(\rho)$.

\begin{lemma}[Higher-order remainder alternative]\label{lem:remainder_alter}
Assume $\tilde{\alpha}(\rho) < 1$ and $\beta(\rho) < 1$ with $\tilde{\alpha},\beta$ as in~\eqref{eq:alpha-beta},~\eqref{eq:alpha_tilde}. Then there is an absolute constant $C > 0$ such that
\begin{equation*}
    \abr{\sum_{\ell=3}^{\infty} I_\ell} \le \frac{Cd}{g}\rbr{2(\lambda_1 - \lambda_k) + \pi g}\rbr{\lambda_1 + \frac{g}{2}}\rbr{\frac{\tilde{\alpha}(\rho)^3}{1 - \tilde{\alpha}(\rho)} + 2 e^{-(n+d)}\,\frac{\beta(\rho)^3}{1 - \beta(\rho)}}.
\end{equation*}
\end{lemma}

\begin{proof}
    Follow the same steps as in Appendix~\ref{app:remainder}.
\end{proof}

\begin{theorem}\label{thm:smallcluster-main_alter}
Assume $\tilde{\alpha}(\rho) < 1$ and $\beta(\rho) < 1$ with $\tilde{\alpha}, \beta$ as in~\eqref{eq:alpha-beta},~\eqref{eq:alpha_tilde}. Then
\begin{equation*}
    \Tcal \ge \sum_{j=k+1}^{d}\nu_j^\downarrow - \frac{C}{g}\sbr{\rho^2 \sum_{i=1}^{k}\sum_{j=k+1}^{d}(\lambda_i + \lambda_j) + \sum_{j=1}^{k}\rbr{\nu_j^\downarrow}^2 + \rho^4 n\, k(d-k)} - \tilde{\Rcal}(\rho),
\end{equation*}
where the higher-order remainder $\tilde{\Rcal}(\rho)$ is the right-hand side of Lemma~\ref{lem:remainder_alter}.
\end{theorem}
\begin{proof}
    Follow the same steps as in Appendix~\ref{app:smallcluster-main}.
\end{proof}

\begin{corollary}\label{cor:smallcluster_alter}
There exist absolute constants $c_1, c_2, c_g > 0$ such that, if
\begin{equation*}
    \rho\sqrt{nd} \le c_1\,\sigma_k, \qquad \sigma_1 \le c_2\,\sigma_k, \qquad g > c_g \lambda_k,
\end{equation*}
then there is a constant $C > 0$ such that
\begin{equation*}
    \Tcal \ge \sum_{j=k+1}^{d}\nu_j^\downarrow - C\rbr{k(d-k)\,\rho^2 + \frac{k\,n^2 \rho^4}{\sigma_k^2} + \frac{d n^{3/2}\rho^3}{\sigma_k}}.
\end{equation*}
\end{corollary}

\begin{proof}

\emph{Verifying the hypotheses of Theorem~\ref{thm:smallcluster-main_alter}.} Using $\sigma_1 \le c_2\sigma_k$, $\rho\sqrt{nd} \le c_1\sigma_k$, and $n \ge d \ge 1$,
\begin{equation*}
    \sigma_1\rho\sqrt n \le \frac{c_2\sigma_k\cdot c_1\sigma_k}{\sqrt d} \le c_1 c_2\,\sigma_k^2,\qquad
    \rho^2 n = \frac{(\rho\sqrt{nd})^2}{d} \le c_1^2\,\sigma_k^2,
\end{equation*}
\begin{equation*}
    \sigma_1\rho\sqrt{nd} \le c_1 c_2\,\sigma_k^2,\qquad
    \rho^2 nd \le c_1^2\,\sigma_k^2.
\end{equation*}
Substituting into~\eqref{eq:alpha_tilde} and~\eqref{eq:alpha-beta}, and using $g \ge c_g\sigma_k^2$,
\begin{equation*}
    \tilde\alpha(\rho) \le \frac{2C(c_1 c_2 + c_1^2)}{c_g},\qquad
    \beta(\rho) \le \frac{2(2 c_1 c_2 + c_1^2)}{c_g}.
\end{equation*}
Choosing $c_1$ sufficiently small, we may ensure $\tilde\alpha(\rho), \beta(\rho) \le 1/2$ so that $\tfrac{1}{1-\tilde\alpha(\rho)},\ \tfrac{1}{1-\beta(\rho)} \le 2$. The hypotheses of Theorem~\ref{thm:smallcluster-main_alter} are thus satisfied.

\emph{Bounding the leading terms.} From $\sigma_1 \le c_2\sigma_k$, $\lambda_1 = \sigma_1^2 = O(\sigma_k^2)$ and hence $\lambda_i + \lambda_j = O(\sigma_k^2)$ for all $i \le k < j$. Therefore
\begin{equation*}
    \frac{\rho^2}{g}\sum_{i=1}^{k}\sum_{j=k+1}^{d}(\lambda_i+\lambda_j) \le C\,k(d-k)\,\rho^2.
\end{equation*}
Since each $\nu_j \le n\rho^2/4$,
\begin{equation*}
    \frac{1}{g}\sum_{j=1}^{k}(\nu_j^\downarrow)^2 \le C\frac{kn^2\rho^4}{\sigma_k^2}, \qquad \frac{\rho^4 n\,k(d-k)}{g} \le C\frac{\rho^4\,n\,k(d-k)}{\sigma_k^2} \le C\frac{kn^2\rho^4}{\sigma_k^2}.
\end{equation*}

\emph{Bounding the remainder.} Plugging~\eqref{eq:alpha_tilde} and~\eqref{eq:alpha-beta} into Lemma~\ref{lem:remainder_alter} together with $g = \Theta(\sigma_k^2)$ and $\lambda_1 = O(\sigma_k^2)$ gives
\begin{equation}\label{eq:R-cube-form-alter}
    \tilde{\Rcal}(\rho) \le \frac{Cd}{\sigma_k^4}\sbr{\rbr{\sigma_1\rho\sqrt n + \rho^2 n}^3 + 2 e^{-(n+d)}\rbr{2\sigma_1\rho\sqrt{nd}+\rho^2 nd}^3}.
\end{equation}
We claim the linear-in-$\rho$ part dominates the quadratic-in-$\rho$ part inside each cube. For the second cube, using $\sigma_1 \ge \sigma_k$ and $\rho\sqrt{nd} \le c_1\sigma_k$,
\begin{equation*}
    \frac{2\sigma_1\rho\sqrt{nd}}{\rho^2 nd} \;=\; \frac{2\sigma_1}{\rho\sqrt{nd}} \;\ge\; \frac{2\sigma_k}{c_1\sigma_k} \;=\; \frac{2}{c_1}.
\end{equation*}
For the first cube, write $\rho\sqrt n = \rho\sqrt{nd}/\sqrt d \le c_1\sigma_k/\sqrt d$. Using $\sigma_1 \ge \sigma_k$ and $d \ge 1$,
\begin{equation*}
    \frac{\sigma_1\rho\sqrt n}{\rho^2 n} \;=\; \frac{\sigma_1}{\rho\sqrt n} \;\ge\; \frac{\sigma_k}{c_1\sigma_k/\sqrt d} \;=\; \frac{\sqrt d}{c_1}.
\end{equation*}
(In contrast to Corollary~\ref{cor:smallcluster}, no separate lower bound on $d/\sqrt n$ is required: the alternative remainder uses $\sqrt n$ rather than $\sqrt d$ inside $\tilde\alpha$, so dominance follows directly from $\rho\sqrt{nd}\le c_1\sigma_k$.) Choosing $c_1 \le 1/2$, both ratios are $\ge 2$, so $(a+b)^3 \le 4 a^3$ for $a$ the linear part and $b$ the quadratic part. Thus $(\sigma_1\rho\sqrt n + \rho^2 n)^3 \le 4(\sigma_1\rho\sqrt n)^3$ and $(2\sigma_1\rho\sqrt{nd} + \rho^2 nd)^3 \le 4(2\sigma_1\rho\sqrt{nd})^3$. Substituting into~\eqref{eq:R-cube-form-alter} and using $\sigma_1 \le c_2\sigma_k$,
\begin{equation*}
    \tilde{\Rcal}(\rho) \le C\sbr{\frac{d\, n^{3/2}\rho^3}{\sigma_k} + 2 e^{-(n+d)}\,\frac{d^{5/2} n^{3/2}\rho^3}{\sigma_k}} \le C\,\frac{d\, n^{3/2}\rho^3}{\sigma_k},
\end{equation*}
where the last step absorbs the exponentially small bad-event contribution: since $d \le n$, $e^{-(n+d)} d^{3/2} \le e^{-n} n^{3/2} \to 0$ as $n \to \infty$.

Substituting these bounds into Theorem~\ref{thm:smallcluster-main_alter} proves the claim.
\end{proof}


\end{document}